\documentclass{article} 

\usepackage{amsmath}
\usepackage{amsthm}
\usepackage{amssymb}
\usepackage{stackengine}

\usepackage{url}
\usepackage{xcolor}

\usepackage[colorlinks=true,urlcolor=blue,linkcolor=blue,plainpages=false,pdfpagelabels]{hyperref}

\newtheorem{theorem}{Theorem}[section]
\newtheorem{definition}[theorem]{Definition}
\newtheorem{notation}[theorem]{Notation}
\newtheorem{proposition}[theorem]{Proposition}
\newtheorem{corollary}[theorem]{Corollary}
\newtheorem{lemma}[theorem]{Lemma}
\newtheorem{remark}[theorem]{Remark}
\newtheorem{example}[theorem]{Example}

\newcommand{\bdfn}{\begin{definition}}
\newcommand{\edfn}{\end{definition}}
\newcommand{\bthm}{\begin{theorem}}
\newcommand{\ethm}{\end{theorem}}
\newcommand{\bprop}{\begin{proposition}}
\newcommand{\eprop}{\end{proposition}}
\newcommand{\bcor}{\begin{corollary}}
\newcommand{\ecor}{\end{corollary}}
\newcommand{\blem}{\begin{lemma}}
\newcommand{\elem}{\end{lemma}}
\newcommand{\bex}{\begin{example}\begin{rm}}
\newcommand{\eex}{\end{rm}\end{example}}

\newcommand{\norm}[1]{\left\| {#1}\right\|}
\newcommand{\N}{{\mathbb N}}
\newcommand{\R}{{\mathbb R}}
\newcommand{\CAT}{{\rm{CAT}(0)}}
\newcommand{\ba}{\begin{array}} \newcommand{\ea}{\end{array}}
\newcommand{\pair}[1]{\langle {#1} \rangle}
\newcommand{\eps}{\varepsilon}

\title{Quantitative translations for viscosity approximation methods in hyperbolic spaces}

\author{Ulrich Kohlenbach${}^{\dagger}$ and Pedro Pinto${}^{\ddagger}$\\[2mm]
	\footnotesize Department of Mathematics, Technische Universit{\"a}t Darmstadt,\\ 
	\footnotesize Schlossgartenstra\ss{}e 7, 64289 Darmstadt, Germany \\
	\footnotesize ${}^{\dagger}$\protect\url{kohlenbach@mathematik.tu-darmstadt.de}\\
	\footnotesize ${}^{\ddagger}$\protect\url{pinto@mathematik.tu-darmstadt.de}
}

\begin{document}

\maketitle

\begin{abstract}
In the setting of hyperbolic spaces, we show that the convergence of Browder-type sequences and
Halpern iterations respectively entail the convergence of their viscosity version with a Rakotch map. We also show that the convergence of a hybrid viscosity version of the Krasnoselskii-Mann iteration follows from the convergence of the Browder type sequence. Our results follow from proof-theoretic techniques (proof mining). From an analysis of theorems due to T. Suzuki, we extract a transformation of rates for the original Browder type and Halpern iterations into rates for the corresponding viscosity versions. We show that these transformations can be applied to earlier quantitative studies of these iterations. From an analysis of a theorem due to H.-K. Xu, N. Altwaijry and S. Chebbi, we obtain similar results. Finally, in uniformly convex Banach spaces we study a strong notion of accretive operator due to Brezis and Sibony and extract an uniform modulus of uniqueness for the property of being a zero point. In this context, we show that it is possible to obtain Cauchy rates for the Browder type and the Halpern iterations (and hence also for their viscosity versions).

\end{abstract}

\section{Introduction}

In this paper we study strongly convergent iteration schemes for nonexpansive 
self mappings $T:C\to C$ of bounded closed convex subsets of Banach spaces and 
 - more generally - complete 
$(W)$-hyperbolic spaces $X$ in the sense of 
\cite{Kohlenbach(05)i}. 

In the framework of Hilbert spaces, Browder \cite{Browder(67)} and Halpern \cite{Halpern(67)} proved strong convergence theorems for implicit and explicit iterations, respectively. Browder's implicit scheme is  
\begin{equation}
y_n(u)=(1-\alpha_n)T(y_n(u))\oplus\alpha_nu, \ \ u\in C,
\end{equation}
while Halpern introduced the explicit iteration scheme
\begin{equation}
w_{n+1}(u)=(1-\alpha_n)T(w_n(u))\oplus\alpha_nu, \ \ w_0(u)=u\in C.
\end{equation}
Here `$\oplus$' is the result of applying the convexity operator $W$ from 
our hyperbolic space (see the next section for details) which is the usual
linear convex combination in the case of normed spaces.
\\[1mm]
The implicit 
schema $(y_n(u))$ was introduced by Browder in \cite{Browder(67)} and shown to be strongly convergent in Hilbert spaces  
to the metric projection of $u$ onto the set $Fix(T)$ of fixed points of $T$ 
(see also \cite{Halpern(67)} for a more elementary proof). 
In \cite{Reich(80)}, Reich extended the convergence of the Browder sequence to uniformly smooth Banach spaces 
where the sequence converges to the unique sunny nonexpansive retraction of 
$C$ onto $Fix(T)$ applied to $u.$
\\[1mm]
For the explicit scheme $(w_n(u))$, Halpern showed the strong convergence to the metric projection of the anchor point $u$ onto the set $Fix(T).$ 
The conditions considered by Halpern prevented the natural choice $\alpha_n=\frac{1}{n+1}$, which was later overcome by Wittmann~\cite{Wittmann(92)}. \cite{ShiojiTakahashi(97)} generalized 
Wittmann's theorem in particular to uniformly smooth Banach spaces. 
In \cite{Xu(02)i,Xu(02)ii}, Xu proved the strong convergence of Halpern type iterations in uniformly smooth Banach spaces under conditions which are 
incomparable to those used by Wittmann but which also allow for 
the choice $\alpha_n=\frac{1}{n+1}.$ 
\\[1mm]
Both schemes have also been considered for families of nonexpansive 
mappings $(S_n)$ instead of the single map $T:$ the sequence implicitly 
defined by 
\begin{equation}\label{SnB}
y_n(u)=(1-\alpha_n)S_n(y_n(u))\oplus\alpha_nu,
\end{equation}
is called the $(S_n)$-Browder sequence with anchor point $u$. The original Browder sequence is \eqref{SnB} for a constant sequence $(S_n)$.
\\[1mm] Likewise, the sequence explicitly defined by 
\begin{equation}\label{SnH}
w_0(u)=u\in C\, \text{ and } w_{n+1}(u)=(1-\alpha_n)S_n(w_n(u))\oplus\alpha_nu,
\end{equation}
is called the $(S_n)$-Halpern iteration (with anchor and starting point $u$). The Halpern iteration~\cite{Halpern(67)} is the particular case of \eqref{SnH} for constant $(S_n)$. 
In \cite{Bauschke(96)}, Bauschke considered a finite sequence of nonexpansive maps (not necessarily commutative but under a condition on the fixed point sets for ordered compositions) to define in a cyclic manner an infinite sequence of nonexpansive maps $(S_n)$ from an initial finite list of maps. In this way, Bauschke's result is a particular instance of convergence for an iteration in the style of \eqref{SnH}. Bauschke's result was further generalized in 
\cite{Jungetal}. 
\\[1mm]
A different instance of convergence for iterations \eqref{SnH} can be found in the methods for finding zeros of accretive operators, the so-called proximal point algorithms. The Halpern-type proximal point algorithm (HPPA), introduced by Kamimura and Takahashi~\cite{KamimuraTakahashi(00)} and independently by Xu~\cite{Xu(02)i}, is the particular case of \eqref{SnH} where the sequence of nonexpansive maps is given by resolvent functions of an accretive operator. The HPPA has been extensively studied in the literature and several results give conditions that guarantee the strong convergence of the algorithm 
both in Hilbert spaces as well as e.g. in Banach spaces which are both 
uniformly smooth and uniformly convex 
(e.g. \cite{KamimuraTakahashi(00),Xu(02)i,BoikanyoMorosanu(11)ii,KhatibzadehRanjbar(13),AoyamaToyoda(17)}).
\\[1mm]
In \cite{Moudafi(00)}, Moudafi introduced in the context of Hilbert spaces the so-called viscosity algorithms in which the 
fixed anchor $u$ is replaced 
by the value of a strict contraction $\phi$ applied 
to the current iteration. 
Moudafi's viscosity algorithms were extended to uniformly smooth Banach spaces by Xu in \cite{Xu(04)}.
While Moudafi only considered a single 
mapping $T$ this has also subsequently been generalized to families of such 
mappings $(S_n):$ the sequence implicitly defined by 
\begin{equation}\label{VSnB}
x_n=(1-\alpha_n)S_n(x_n)\oplus\alpha_n\phi(x_n),
\end{equation}
is called the $(S_n)$-viscosity-Browder sequence (for $(\alpha_n)$).
\\[1mm]
The $(S_n)$-viscosity-Halpern iteration is an explicit counterpart of \eqref{VSnB}, namely, given $x_0\in C$, the iteration is defined inductively by
\begin{equation}\label{VSnH}
x_{n+1}=(1-\alpha_n)S_n(x_n)\oplus\alpha_n\phi(x_n).
\end{equation}
Viscosity generalizations of Bauschke's result for finite families 
of mappings and of the results in \cite{Jungetal} are given in 
\cite{Jung(06)} and - under very general conditions - in 
\cite{Chang(06)} which is closely related to the prior 
Hybrid Steepest Descent Method, introduced by Yamada in \cite{Yamada(01)} 
for finite sets of nonexpansive functions. In fact, the proof 
in \cite{Yamada(01)} can be adapted to provide a strong convergence result 
for the viscosity version of Bauschke's theorem (see \cite{Koernlein(diss)}). 
\\[1mm]
In the important paper \cite{Suzuki(07)}, Suzuki showed in the context 
of normed spaces that under very general conditions the convergence of 
the Browder and Halpern schemes (for families of mappings) implies the 
convergence of the corresponding viscosity schemes. Moreover, he showed 
that the limit of $(x_n)$ is the unique fixed point of $P\circ \phi,$ 
where $P(u):=\lim y_n(u)$ in the Browder case and 
$P(u):=\lim w_n(u)$ in the Halpern case respectively. Suzuki's result 
actually applies to a larger class of mappings $\phi$  than only strict 
contractions, namely to so-called Meir-Keeler contractions (MKC) introduced 
in \cite{MeirKeeler(69)} which in the case of convex sets 
(even of hyperbolic spaces) coincide with the uniformly contractive mappings 
introduced earlier by Rakotch in \cite{Rakotch(62)}. With $\phi,$ also 
the map $x\mapsto (1-\alpha_n)S_n(x)\oplus \alpha_n\phi(x)$ is an MKC 
mapping and so -- by \cite{MeirKeeler(69),Rakotch(62)} -- has a unique 
fixed point. Thus the implicit Browder scheme \eqref{VSnB} (even for $\phi$ a MKC) is well-defined.
\\[1mm]
In section 3 of this paper, we give a complete quantitative analysis of 
Suzuki's reduction technique both in terms of rates of convergence as well as 
in terms of rates of metastability in the sense of Tao. Since in many important 
cases, explicit rates of metastability have been construced in recent years 
for both Browder-type sequences as well as for Halpern iterations, we now 
get in all these situations also rates of metastability for the corresponding 
viscosity generalizations. Usually, effective rates of convergence can be ruled 
out for Browder and Halpern sequences (see e.g. \cite{Neu15}) 
and so it is important that our 
quantitative analysis of Suzuki's theorem allows us to operate on the 
level of rates of metastability, by which we mean for a sequence $(x_n)$ 
in a metric space $(X,d)$ any bound $\varphi:(0,\infty)\times \N^{\N}\to 
\N$ such that 
\[ \forall \varepsilon>0\,\forall f:\N\to\N\,\exists n\le\varphi(\varepsilon,f) 
\,\forall i,j\in [n,f(n)]\, \left( d(x_i,x_j)\le\varepsilon\right) \]
(see \cite{Kohlenbach(computational),Tao(08)}). Note that noneffectively, 
the property bounded by $\varphi$ is equivalent to the usual Cauchy property.
So a rate of metastability provides a finitary quantitative statement which 
(noneffectively) implies back the convergence statement. Moreover, we 
augment the rate of metastability with further bounding information which 
entails that the limits of $(x_n)$ is the unique fixed point of $P\circ 
\phi$ with $P$ as above (see Theorems \ref{SBtheo} and \ref{SHtheo}).
\\[1mm] Our analysis uses ideas from the logic-based approach of `proof 
mining' by which the extractability of such quantitative data follows 
from certain logical transformations of the given proof 
(see \cite{Koh2008} for more on this). However, while ideas from 
logic played a key role in arriving at our results, the final proofs 
make no references to logic. As a common 
by-product of such a logical analysis of proofs one often obtains in 
addition to quantitative information also new qualitative generalizations 
of the original theorems. E.g. the generalization to the setting of hyperbolic spaces is such an offspring of the logic-based 
approach.
\\[1mm]
If one is in a special situation where an actual rate of convergence for the 
Browder of Halpern sequences is available, then our quantitative 
transformation also yields a rate of convergence for the corresponding 
viscosity version.
\\[1mm]
A different method  for finding fixed points of nonexpansive mappings is the Krasnoselskii-Mann iteration
\begin{equation}\label{KM}
x_{n+1}=(1-\beta_n) x_n \oplus\beta_nT(x_n).
\end{equation}

This iteration has nice properties, notably Fej\'er monotonicity w.r.t. 
the fixed point set of $T,$ but in general only approximates a fixed point weakly. This motivated several modifications to \eqref{KM} to ensure strong convergence. Here we consider a hybrid version of the Krasnoselskii-Mann iteration and the viscosity approximation method. Namely, for $(\beta_n)\subset(0,1]$ and $x_0\in C$, the \rm{vKM} iteration is defined by 
\begin{equation}\label{vKM}
x_{n+1}=(1-\beta_n)x_n\oplus \beta_n\left((1-\alpha_n)T(x_n)\oplus \alpha_n \phi( x_n)\right),
\end{equation}
introduced in \cite{Gwinner(78)} and shown there 
to be strongly convergent in Banach spaces under very strong conditions. 
In \cite{Xu(2020)} this is studied further and much improved. 
Note that for $\alpha_n\equiv 0$, \eqref{vKM} reduces to the original schema \eqref{KM}, and for $\beta_n\equiv 1$, \eqref{vKM} is the viscosity generalization of the Halpern iteration.
\\ Again we provide a quantitative analysis of the results 
from \cite{Gwinner(78),Xu(2020)}.
\\[1mm] 
In the final section 4 we study a situation in which explicit rates of 
convergence (rather than only rates of metastability) for 
Browder-type sequences, Halpern iterations as well as  
Krasnoselskii-Mann 
iterations can be obtained. Here the 
nonexpansive mapping $T:C\to C$ used in these sequences is of the form 
$I-A,$ where $A$ is an accretive operator which satisfies a condition 
of being uniformly accretive which was introduced in \cite{BrezisSibony} 
(see also \cite{Gwinner(78)} and \cite{Xu(2020)}) and which is more general 
than other notions of uniform $\Psi$ or $\Phi$-accretivity studied in 
the literature. In uniformly convex Banach spaces $X$, this liberal notion 
suffices to conclude that $A$ has at most one zero. From this uniqueness 
proof due to \cite{BrezisSibony} we extract a so-called modulus of uniqueness 
(see \cite{Koh2008} and the references given there) 
which only depends on a general modulus function witnessing the 
uniform accretivity of $A$ and a modulus of uniform convexity of $X.$
From this we obtain explicit and low complexity rates of convergence for the 
iterations listed above involving the mapping $T=I-A.$

\section{Preliminaries}

\subsection{Classes of maps and hyperbolic spaces}
Let $(X,d)$ be a metric space and $C$ a nonempty subset of $X$.
\begin{definition}
	Consider a mapping $T:C\to C$. We say
	\begin{itemize}
		\item $T$is nonexpansive if for all $x,y\in C$
		\begin{equation*}
		d(T(x),T(y))\leq d(x,y)\, ;
		\end{equation*}
		\item $T$ is a strict contraction if there is $r\in [0,1)$ such that for all $x,y\in C$
		\begin{equation*}
		d(T(x),T(y))\leq rd(x,y),
		\end{equation*}
		in which case we say that $T$ is an $r$-contraction;
		\item  $T$ is a Meir-Keeler contraction (MKC) if for any $\varepsilon >0$ there is $\delta >0$ such that
		\begin{equation*}
		d(x,y)< \varepsilon + \delta \to d(T(x),T(y))< \varepsilon, \text{ for all }\, x, y\in C.
		\end{equation*}
	\end{itemize}
\end{definition}
Clearly, any contraction is a MKC mapping and any MKC mapping is a nonexpansive map. Meir-Keeler contraction were introduced in \cite{MeirKeeler(69)} as a generalization of metric contractions. Actually, \cite{MeirKeeler(69)} considers a condition with the stronger premise $\eps\leq d(x,y)< \eps+\delta$, which ends up being equivalent to the one above using the nonexpansivity of $T$.

We now recall the notion of hyperbolic space. There are several distinct notion of hyperbolic space in the literature \cite{Kirk(82),GoebelKirk(83),GoebelReich(84),ReichShafrir(90)}. Our results will be in the setting of $W$-hyperbolic spaces as introduced by the first author \cite{Kohlenbach(05)i}, here labeled simply by \emph{hyperbolic spaces}. This setting is slightly more restrictive than the notion of hyperbolic space by Goebel/Kirk~\cite{GoebelKirk(83)}, but more general than the hyperbolic spaces in the sense of Reich/Shafrir~\cite{ReichShafrir(90)}.

The triple $(X,d,W)$ is called a hyperbolic space if $(X,d)$ is a metric space and $W: X\times X \times [0,1] \to X$ is a function satisfying
\begin{enumerate}
	\item[(W1)] $\forall x, y\in X \forall \lambda \in [0,1] \, \left( d(z, W(x,y,\lambda))\leq (1-\lambda)d(z,x)+\lambda d(z,y) \right)$,
	\item[(W2)] $\forall x, y\in X \forall \lambda_1, \lambda_2 \in [0,1] \, \left( d(W(x,y,\lambda_1), W(x,y,\lambda_2)=|\lambda_1-\lambda_2|d(x,y)) \right)$,
	\item[(W3)] $\forall x, y\in X \forall \lambda \in [0,1] \, \left( W(x,y,\lambda)=W(y,x, 1-\lambda) \right)$,
	\item[(W4)] $\begin{cases}\forall x, y, z, w\in X \forall \lambda \in [0,1]\\
	\qquad\left( d(W(x,z,\lambda), W(y,w,\lambda))\leq (1-\lambda)d(x,y)+\lambda d(z,w) \right).\end{cases}$
\end{enumerate}

The convexity function $W$ was first considered by Takahashi \cite{Takahashi(70)}, where a triple $(X,d,W)$ with $W$ satisfying (W1) is called a convex metric space. The classe of hyperbolic spaces includes normed spaces and their convex subsets, the Hilbert ball \cite{GoebelReich(84)} and \CAT-spaces in the sense of Gromov. (see e.g. \cite{BridsonHaefliger(13)} for a detailed treatment). In turn, hyperbolic spaces are \CAT-spaces if they satisfy the property CN$^-$ (which in the presence of other axioms is equivalent to the Bruhat-Tits CN-inequality~\cite{BruhatTits(72)} but in contrast 
to the latter purely universal):
\begin{equation*}
	\text{CN}^-:\quad\forall x, y, z\in X\, \left( d^2(z, W(x,y, \frac{1}{2}))\leq \frac{1}{2}d^2(z,x)+\frac{1}{2}d^2(z,y)-\frac{1}{4}d^2(x,y)\right).
\end{equation*}

If $x, y \in X$ and $\lambda \in [0,1]$, we shall denote $W(x,y,\lambda)$ by $(1-\lambda)x \oplus \lambda y$. Using (W1), it is easy to see that
\begin{equation*}
d(x,(1-\lambda)x \oplus \lambda y)=\lambda d(x,y) \, \text{ and } \, d(y,(1-\lambda)x \oplus \lambda y)=(1-\lambda)d(x,y).
\end{equation*}
For $x,y\in X$, the set $\{(1-\lambda)x \oplus \lambda y\, :\, \lambda \in[0,1]\}$ is called the metric segment with endpoints $x,y$ and is denote by $[x,y]$. A nonempty subset $X\subseteq X$ is called \emph{convex} if $[x,y] \subseteq C$, for all $x, y \in C$.

For convex $C$, MKC have the useful property that for each $\varepsilon >0$ there is a contraction factor for points $\varepsilon$-apart. This was 
shown for convex subsets of Banach spaces in \cite[Proposition 2]{Suzuki(07)}. We generalize
this result to the geodesic setting: 
\begin{lemma}\label{MKcontractionR}
	Let $C$ be a convex subset of $X$ and $\phi$ a MKC on $C$. Then there is a function $\delta:(0, \infty)\to (0,1)$ such that for all $\eps>0$ and $x, y\in C$
	\begin{equation*}
		d(x,y)\geq \varepsilon \to d(T(x), T(y))\leq (1-\delta(\eps))\cdot d(x,y).
	\end{equation*}
\end{lemma}

\begin{proof}
	Let $\eps>0$ be given. Since $\phi$ is a MKC mapping, there is $\sigma\in(0,\eps)$ satisfying for all $x,y \in C$
	\[d(x,y)< \frac{\eps}{4}+\sigma \to d(\phi(x), \phi(y))\leq \frac{\eps}{4}.\]
	
	We claim that for all $x, y\in C$,
	\begin{equation}\label{multfactor}
		d(x,y)\geq \eps \to d(\phi(x),\phi(y))\leq \left(1-\frac{\sigma}{4\eps}\right)\cdot d(x,y),
	\end{equation}
and thus for each $\eps>0$ we can define $\delta(\eps)$ by
\[\frac{1}{4\eps}\sup\left\{\sigma\in(0,\eps): \forall x, y\in C\,\left(d(x,y)\geq \eps \to d(\phi(x),\phi(y))\leq (1-\frac{\sigma}{4\eps})\cdot d(x,y)\right) \right\}.\]

	It remains to prove \eqref{multfactor}. Take $x, y \in C$ such that $d(x,y)\geq \eps$ and write $a:=d(x,y)$. For each $j\in \{0, \cdots, \lfloor\frac{a}{\eps}\rfloor+1\}$, consider $\lambda_j\in[0,1]$ defined by $\left(\frac{\eps+\sigma}{4a}\right) j$. For $j\in\{0, \cdots \lfloor\frac{a}{\eps}\rfloor\}$, by (W2)
	\[
	d(W(x,y,\lambda_j), W(x,y,\lambda_{j+1}))=(\lambda_{j+1}-\lambda_j)d(x,y)< \frac{\eps}{4}+\sigma,
	\]
	and hence, by the hypothesis on $\sigma$,
	\[
	d(\phi(W(x,y,\lambda_j)), \phi(W(x,y,\lambda_{j+1})))\leq \frac{\eps}{4}.
	\]
	
	Since $\phi$ is nonexpansive, we have
	\begin{align*}
		d(\phi(x),\phi(y))&\leq \sum_{j=0}^{\lfloor\frac{a}{\eps}\rfloor}d(\phi(W(x,y,\lambda_j)), \phi(W(x,y,\lambda_{j+1})))\\
		&\qquad\qquad\qquad\qquad + d(\phi(W(x,y,\lambda_{\lfloor\frac{a}{\eps}\rfloor+1})), \phi(y))\\
		&\leq \frac{\eps}{4}\left(\left\lfloor\frac{a}{\eps}\right\rfloor+1\right)+(1-\lambda_{\lfloor\frac{a}{\eps}\rfloor+1})d(x,y)\\
		&=a-\frac{\sigma}{4}\left( \left\lfloor\frac{a}{\eps}\right\rfloor+1 \right)\leq a\left(1-\frac{\sigma}{4\eps}\right)=\left(1-\frac{\sigma}{4\eps}\right)d(x,y),
	\end{align*}
showing \eqref{multfactor} and concluding the proof.
\end{proof}

In \cite{Rakotch(62)}, Rakotch generalized Banach's contraction principle to maps satisfying the condition
\begin{equation}\label{rakotchoriginal}
	x\not= y\to d(\phi(x), \phi(y))\leq \alpha(d(x,y))\cdot d(x,y),
\end{equation}
where $\alpha:(0, \infty)\to(0,1)$ is a decreasing function. However, the assumption that $\alpha$ is decreasing implies that condition \eqref{rakotchoriginal} is equivalent to
\begin{equation}\label{rakotchmodified}
	d(x,y)\geq \eps \to d(\phi(x), \phi(y))\leq \alpha(\eps)\cdot d(x,y).
\end{equation}
We may always assume that $\alpha$ in \eqref{rakotchmodified} is decreasing, switching if necessary to
\[\alpha'(\eps):=\max\{ 1/2, \inf\{ \alpha(\delta) : \delta\in(0,\eps]\}\}.\]
This argument shows that \cite[Proposition 2]{Suzuki(07)} actually entails that in normed spaces MKC restricted to convex sets are Rakotch maps. 
The converse is also true: take $\delta:=(1-\alpha(\eps))\cdot\eps$ 
to satisfy the MKC-condition. \\ 
Our Lemma~\ref{MKcontractionR} generalizes this result to the setting of hyperbolic spaces. In our results, $C$ is always convex and so MKC are only seemingly more general than Rakotch maps. In what follows, we call a function $\delta:(0,\infty)\to (0,1)$ a Rakotch-modulus (see \cite{KohlenbachOliva(03)}) if it satisfies
\begin{equation*}
	\forall \varepsilon >0\, \forall x,y \in C\, \left( d(x,y)\geq \varepsilon \to d(\phi(x), \phi(y))\leq (1-\delta(\varepsilon))d(x,y)\right)
\end{equation*}

The particular case where $\phi$ is an $r$-contraction coincides with the case where there is a constant Rakotch-modulus (i.e. for all $\varepsilon\in(0,\infty$,  $\delta(\varepsilon)=\delta$, where $\delta\in (0,1)$ is such that $r\leq 1-\delta$).\\ {\bf Notational convention:} Throuhgout this paper 
$\lceil x\rceil$ is defined as $\max\{ 0, \lceil x\rceil\}$ with the usual 
definition of $\lceil\cdot\rceil.$

\subsection{Quantitative notions}
Let $(x_n)$ be a Cauchy sequence in a metric space $(X,d)$.
\begin{definition}
A function $\rho:(0,\infty)\to\N$ is a Cauchy rate of $(x_n)$ if 
\[ \forall\varepsilon >0\,\forall i,j\ge \rho(\varepsilon)\,\left( 
d(x_i,x_j)\le \varepsilon\right). \]
\end{definition} 
If $(X,d)$ is complete so that $(x_n)$ converges, then clearly $\rho$ is 
a rate of convergence of $(x_n)$ towards its limit $x,$ i.e. 
\[ \forall \varepsilon>0\,\forall n\ge \rho(\varepsilon)\,\left( d(x_n,x)\le \varepsilon\right).\]
\begin{definition}[\cite{Kohlenbach(computational),Tao(08)}] 
A function $\varphi:(0,\infty)\times\N^{\N}\to \N$ is a rate of metastability 
for $(x_n)$ if 
\[ \forall \varepsilon>0\,\forall f:\N\to\N\,\exists n\le\varphi(\varepsilon,f) 
\,\forall i,j\in [n,f(n)]\, \left( d(x_i,x_j)\le\varepsilon\right). \]
\end{definition}
\begin{remark}
In the official definition of `rate of stability', one states that
$\forall i,j\in [n,n+f(n)]\, \left( d(x_i,x_j)\le\varepsilon\right)$
which clearly implies by the formulation above and which in turn is 
implied by the latter taking $\tilde{f}(n):=n+f(n)$ instead of $f.$
\end{remark}
\begin{proposition}\label{prop-rate}
$\rho:(0,\infty)\to (0,\infty)$ is a Cauchy rate iff 
$\varphi(\varepsilon,f):=\rho(\varepsilon)$ is a rate 
of metastability.
\end{proposition}
\begin{proof} 
If $\rho$ is a Cauchy rate, then clearly $\varphi(\varepsilon,f):=\rho(\varepsilon)$
is a rate of metastability. Let $\varphi(\varepsilon,f)$ be a rate of metastability which is 
a function $\rho(\varepsilon)$ which does not depend on $f.$ Let $\tilde{m}\ge 
m\ge \rho(\varepsilon).$ 
Take $f(k):=\tilde{m}$ for all $k.$ Then there exists an $n\le \rho(\varepsilon)$ such that 
\[ \forall i,j\in [n,f(n)]\subseteq [\rho(\varepsilon),\tilde{m}]\, (d(x_i,x_j)\le \varepsilon) \]
and so, in particular, $d(x_{\tilde{m}},x_m)\le \varepsilon.$ 
\end{proof}
\subsection{Versions of a Lemma by Xu}

A well-known lemma by Xu (e.g. \cite[Lemma 2.1]{Xu(02)ii}) is used in the original proofs.
\begin{lemma}\label{LemmaXu}
	Consider sequences of nonnegative real numbers $(\lambda_n)\subseteq (0,1)$, $(a_n)$ and a sequence $(b_n)\subset\R$ such that $a_{n+1}\leq (1-\lambda_n)a_n+\lambda_nb_n$, for all $n\in \N$. If 
\[ (i) \ \sum^{\infty}_{n=0}\lambda_n=\infty\,\mbox{(or equivalently $\prod^{\infty}_{n=0} (1-\lambda_n)=0)$}, \ 
		(ii) \ \limsup b_n\leq 0,\]
	then $\lim a_n=0$.
\end{lemma}

We say that $A:\N\to\N$ is a rate of divergence for $\sum\lambda_n=\infty$ if
\begin{equation}\label{condsum}\tag{Q$\Sigma$}
\forall k\in \N\, \left( \sum_{i=0}^{A(k)}\lambda_i \geq k\right).
\end{equation}
A quantitative treatment is also possible using instead the equivalent condition $\prod (1-\lambda_n)=0$, if we have a function $A':\N\times (0,1]\to \N$ such that for all $m\in \N$, $A'(m,\cdot)$ is a rate of convergence for $\prod_{i=m}^{\infty}(1-\lambda_i)=0$, i.e.
\begin{equation}\label{condprod}\tag{Q$\Pi$}
\forall m\in \N \forall \varepsilon\in (0,1]\, \left(\prod_{i=m}^{A'(m,\varepsilon)}(1-\lambda_i)\leq \varepsilon\right).
\end{equation}

For our analysis, we will need a quantitative version of this lemma, which we give bellow. The proof is identical to several similar previous results (see e.g. \cite[Lemmas 5.2,5.3]{KohlenbachLeustean(12)} and \cite[Lemmas 12,13]{Pinto(ta)}).
\begin{lemma}\label{Hlem1}
	Consider sequences of real numbers $(\lambda_n)\subset [0,1]$, 
$(a_n),(b_n)\subset\R$ and let $B\in\N^*$ be an upper bound on $(a_n)$ and assume that for all $n\in\N$
	\[a_{n+1}\leq (1-\lambda_n)a_n+\lambda_nb_n.\]
	Consider $\varepsilon >0, p,N\in \N$ such that $\forall i\in[N,p]\, \left(b_n\leq \frac{\varepsilon}{2}\right)$.
	\begin{itemize}
		\item[(i)] If $A:\N\to \N$ satisfies \eqref{condsum}, then
		\[\forall i\in[\sigma_1(\varepsilon, N), p]\, \left(a_i\leq \varepsilon\right),\]
		where $\sigma_1(\varepsilon, N)=\sigma_1[A, B](\varepsilon, N):=A\left(N+\lceil\ln\left(\frac{2B}{\varepsilon}\right)\rceil\right)+1$;
		\item[(ii)] If $A':\N\times (0,1]\to \N$ satisfies \eqref{condprod}, then
		\[\forall i\in[\sigma_2(\varepsilon, N), p]\, \left(a_i\leq \varepsilon\right),\]
		where $\sigma_2(\varepsilon, N)=\sigma_2[A',B](\varepsilon, N):=\max\{ A'(N, \frac{\varepsilon}{2B}), N\}+1$.
	\end{itemize}
\end{lemma}

The conclusion of Lemma~\ref{Hlem1} still holds true if the main inequality is only valid in the interval $[N,p]$.
\begin{lemma}\label{Hlem1_restricted}
	Consider sequences of real numbers $(\lambda_n)\subset [0,1]$, 
$(a_n),(b_n)\subset\R.$  Let $B\in\N^*$ be an upper bound on $(a_n)$ and $A:\N\to \N$ a function satisfying \eqref{condsum}. Let $\varepsilon\in(0,1]$ and $p,N\in \N$ be given. If
	\[\forall i\in[N,p]\, \left(a_{i+1}\leq (1-\lambda_i)a_i+\lambda_ib_i \
\mbox{and}\ b_i\leq \frac{\varepsilon}{2}\right),\]
	then
	\[\forall i\in[\sigma_1(\varepsilon, N), p]\, \left(a_i\leq \varepsilon\right),\]
	where $\sigma_1(\varepsilon, N)=\sigma_1[A, B](\varepsilon, N):=A\left(N+\lceil\ln\left(\frac{2B}{\varepsilon}\right)\rceil\right)+1$.
\end{lemma}

\section{Viscosity approximation methods}\label{main}

\subsection{Stability conditions}

In \cite{Suzuki(07)}, Suzuki starts by giving perspicuous proofs of Xu's strong convergence results for the viscosity method (with strict contractions) of Browder and Halpern type iterations in uniformly smooth Banach spaces.   
\begin{theorem}[{\cite[Thm.4.1]{Xu(04)}}]
	Let $X$ be an uniformly smooth Banach space and $C$ a nonempty, closed, convex subset of $X$. Let $T$ and $\phi$ be a nonexpansive and a strict contraction on $C$, respectively. Assume that $Fix(T)\not=\emptyset.$ Consider the net $(x_{\alpha})$ satisfying
	\begin{equation*}
	x_{\alpha}=(1-\alpha)T(x_{\alpha}) + \alpha\phi(x_{\alpha}),
	\end{equation*}
	for $\alpha\in (0,1)$. Then $(x_{\alpha})$ strongly converges, as $\alpha\to 0$, to the unique $z$ satisfying $P\phi(z)=z$, where $P$ is the unique sunny retraction of $C$ onto $Fix(T)$.
\end{theorem}

\begin{theorem}[{\cite[Thm.4.2]{Xu(04)}}] Under the same conditions as before, 
let $x_0\in C$, and $(\alpha_n)\subset (0,1)$ satisfy
\[ (C1) \ \lim\alpha_n=0, \ (C2) \ \sum \alpha_n=\infty,
		\ (C3) \ \lim\frac{\alpha_{n+1}}{\alpha_n}=1. \]
	Consider the the iteration $(x_n)$ defined by
\begin{equation*}
x_{n+1}=(1-\alpha_n)T(x_n) + \alpha_n\phi(x_n).
\end{equation*}
Then $(x_n)$ strongly converges to the unique $z$ satisfying $P\phi(z)=z$.
\end{theorem}

The new proofs compare the viscosity iteration with the convergent original one. Suzuki, then goes on to employ similar arguments to obtain extended convergence results to MKC-mappings $\phi$ and with $T$ replaced by a sequence of nonexpansive maps $(S_n)$ together with a stability condition that the original iteration defined on that sequence of maps is convergent. From a logical point of view, the only complicated aspect in Suzuki's proofs is the use of the sunny retraction (metric projection in Hilbert spaces) required to characterize the limit point $z$. We avoid this step by considering Cauchy points of the relevant iterations instead of the actual limit. This entails that a quantitative analysis of these results will extract a function that given a metastability rate for the assumed convergence (and some other parameters that illustrate the uniformity of the bound), outputs a rate of metastability for the the viscosity iteration in question. In this section, we elaborate on the quantitative conditions of our main results given in the next two sections. We then apply our main theorems in particular cases where metastability rates have been extracted and witness the assumption, to obtain several rates of metastability for viscosity-type iterations.

\begin{notation}\label{notation_parameters}
	Consider a function $\varphi$ on tuples of variables $\bar{x}$, $\bar{y}$. If we wish to consider the variables $\bar{x}$ as parameters we write $\varphi[\bar{x}](\bar{y})$. For simplicity of notation we may then even omit the parameters and simply write $\varphi(\bar{y})$.
\end{notation}

The next lemma states that we can always transform a given 
rate of metastability such that the interval of metastability is to the right 
of a given $N\in\N.$
\begin{lemma}\label{lowerbound_meta}
	If $(x_n)$ is a Cauchy sequence with rate of metastability $\varphi$, then
	\begin{equation*}
	\forall \varepsilon >0\, \forall f:\N\to \N\, \forall N\in\N\, \exists n\in [N, \theta(\varepsilon, f, N)] \forall i, j \in [n,f(n)]\, \left( d(x_i,x_j)\leq \varepsilon\right),
	\end{equation*}
	where $\theta(\varepsilon, f, N):=\theta[\varphi](\varepsilon, f, N):=\max\{N, \varphi(\varepsilon, f_N)\}$, where for all $m\in\N$ $f_N(m):=
f(\max\{N,m \})$ .
\end{lemma}

\begin{proof}
	Let $\varepsilon, f$ and $N$ be given. Since $\varphi$ is a rate of metastability for $(x_n)$, we have that for some $n'\leq \varphi(\varepsilon, f_N)$,
	\begin{equation*}
	\forall i, j \in [n', f_N(n')]\, \left(d(x_i, x_j)\leq \varepsilon\right).
	\end{equation*}
	Let $n=\max\{N, n'\}$. Then $n\in[N, \max\{N, \varphi(\varepsilon, f_N)\}]=[N, \theta(\varepsilon, f, N)]$. Since $n'\leq n$ and $f_N(n')=f(n)$, the result follows.
\end{proof}
Let $C\subseteq X$ be a bounded convex subset of a hyperbolic space $X$ and 
$\N^*\in b\ge diam(C).$ \\ 
For $u\in C$ let $(y_n(u))$ be the $(S_n)$-Browder sequence with anchor point $u$ satisfying \eqref{SnB}. We say that a \emph{monotone} function $\theta_b$ satisfies the condition \eqref{BSncondition} if for all $u\in C$,
\begin{equation}\label{BSncondition}\tag{B[$S_n$]}
\begin{split}
\forall \varepsilon >0\, \forall f:&\N\to \N\, \forall N\in\N\, \exists n\in [N, \theta_b(\varepsilon, f, N)]\\
&\forall i, j \in [n,f(n)]\, \left( d(y_i(u),y_j(u))\leq \varepsilon\right).
\end{split}
\end{equation}

\begin{remark}\label{remark1}
	Note that the function $\theta_b$ is uniform on $u\in C$, depending only on the bound $b\in\N^*$ for the diameter of $C$. Furthermore, the monotonicty condition here is to be understood in the sense that the function $\theta_b$ is nondecreasing in $N$. Such assumption is by convenience, as one can always consider $\theta_b^{maj}(\varepsilon, f, N):=\max_{N'\leq N}\{\theta_b(\varepsilon, f, N')\}$. Finally, we remark that by {\rm Lemma~\ref{lowerbound_meta}}, such a function $\theta_b$ exists whenever we have a uniform metastability rate $\varphi_b$ for $(y_n(u))$.
\end{remark}
	
Similarly, for $u\in C$ let $(w_n(u))$ be the $(S_n)$-Halpern iteration with anchor and starting point $u$ defined by \eqref{SnH}. We say that a monotone function $\theta'_b$ satisfies the condition \eqref{HSncondition} if for all $u\in C$,
\begin{equation}\label{HSncondition}\tag{H[$S_n$]}
\begin{split}
\forall \varepsilon >0\, \forall f:&\N\to \N\, \forall N\in\N\, \exists n\in [N, \theta'_b(\varepsilon, f, N)]\\
&\forall i, j \in [n,f(n)]\, \left( d(w_i(u),w_j(u))\leq \varepsilon\right).
\end{split}
\end{equation}

We also write the previous conditions for constant sequences $(S_n)$ to discuss particular instances of the main theorems. Let $T$ be a nonexpansive map.

A monotone function $\theta_b$ satisfies the condition \eqref{Bcondition} if for all $u\in C$,
\begin{equation}\label{Bcondition}\tag{B[$T$]}
\begin{split}
\forall \varepsilon >0\, \forall f:&\N\to \N\, \forall N\in\N\, \exists n\in [N, \theta_b(\varepsilon, f, N)]\\
&\forall i, j \in [n,f(n)]\, \left( d(y_i(u),y_j(u))\leq \varepsilon\right),
\end{split}
\end{equation}
where $(y_n(u))$ is the $(S_n)$-Browder sequence for $S_n\equiv T$, i.e. the original Browder type iteration.

A monotone function $\theta'_b$ satisfies the condition \eqref{Hcondition} if for all $u\in C$,
\begin{equation}\label{Hcondition}\tag{H[$T$]}
\begin{split}
\forall \varepsilon >0 \forall f:&\N\to \N \forall N\in\N \exists n\in [N, \theta'_b(\varepsilon, f, N)]\\
&\forall i, j \in [n,f(n)]\, \left( d(w_i(u),w_j(u))\leq \varepsilon\right),
\end{split}
\end{equation}
where $(w_n(u))$ is the $(S_n)$-Halpern iteration for $S_n\equiv T$, i.e. the original Halpern type iteration.

\subsection{Browder-viscosity}

In this section we prove a quantitative version of Theorem 7 in 
\cite{Suzuki(07)}.

\begin{theorem}\label{SBtheo}
	Let $C$ be a nonempty bounded convex subset of a hyperbolic space and let $b\in \N^*$ be a bound on the diameter of $C$. Consider $(S_n)$ a family of nonexpansive maps on $C$, an MKC mapping $\phi:C\to C$  with Rakotch-modulus $\delta$, and a sequence $(\alpha_n)\subset (0,1]$. Let $\theta_b$ be a monotone function satisfying \eqref{BSncondition}. Define 
	\begin{equation*}
	\Psi(\varepsilon, f,N)=\Psi[b, \delta, \theta_b](\varepsilon, f,N)=\theta_b\left(\varepsilon_0, f, \Psi_M\right),
	\end{equation*}
	where
	\begin{equation*}
	\begin{gathered}
	\Psi_0=N \text{ and } \Psi_{m+1}=\theta_b\left(\varepsilon_0, f_{M-m},
\Psi_m\right), \text{ with }\\
	\tilde{\varepsilon}=\frac{\varepsilon\delta(\varepsilon/2)}{4}, \, \varepsilon_0=\frac{\tilde{\varepsilon}\delta(\tilde{\varepsilon})}{2}, \, M=\left\lceil\log_{1-\delta(\tilde{\varepsilon})}\left(\frac{\tilde{\varepsilon}}{2b}\right)\right\rceil \text{ and}\\
	\text{for all } p\in\N,\, \begin{cases}
	f_0(p)=f(p),\\
	f_{m+1}(p)=\max\left\{f(p),\, \theta_b\left(\varepsilon_0, f_m, p\right)\right\} & \text{for } m<M.
	\end{cases}
	\end{gathered}
	\end{equation*}
Then for $(x_n)$ satisfying \eqref{VSnB} we have
\[ \begin{array}{l}\forall N\in\N\,\forall \varepsilon>0\,\forall f:\N\to\N\,\exists n\in [N,
\Psi(\varepsilon,f,N)] \ \exists z\in C\\[1mm] \hspace*{1cm} 
\left( d(z,y_n(\phi(z)))\le\varepsilon\wedge\forall i\in [n,f(n)] \,(
d(x_i,z)\le \frac{\varepsilon}{2})\right). \end{array}\]
In particular, $\Psi(\varepsilon,f,0)$ is a rate of metastability for 
$(x_n).$
\end{theorem}

\begin{proof}
	Let $N\in \N$, $\varepsilon >0$ and $f:\N\to \N$ be given. Note that the definition of the natural number $M$ implies $\left(1-\delta(\tilde{\varepsilon})\right)^M\leq \frac{\tilde{\varepsilon}}{2b}$. Consider the functions $f_0, \cdots, f_M$ as in the theorem.  
	Define $z_0=x_0$ and $n_0=N$. For $m\leq M$, assume that $z_m$ and $n_m$ are already defined and let $(y_n^m)_n$ be the $(S_n)$-Browder sequence with anchor point $\phi(z_m)$, i.e. $y_n^m=y_n(\phi(z_m))$. By \eqref{BSncondition}, we can take $n_{m+1}\in[n_m, \theta_b(\varepsilon_0, f_{M-m}, n_m)]$ such that
	\begin{equation*}
	\forall i\in [n_{m+1}, f_{M-m}(n_{m+1})]\, \left(d(y^m_i,y^m_{n_{m+1}})\leq \varepsilon_0\right).
	\end{equation*} 
	Define $z_{m+1}=y^m_{n_{m+1}}$.
	
	Notice that for $m\leq M-1$, we have $n_{m+2}\in [n_{m+1}, \theta_b(\varepsilon_0, f_{M-m-1}, n_{m+1})]\subset [n_{m+1}, f_{M-m}(n_{m+1})]$, and so
	\begin{equation*}
	d(y^m_{n_{m+2}},y^m_{n_{m+1}})\leq \varepsilon_0.
	\end{equation*}
	
	For $m\in [1, M]$, by (W4) and the fact that $S_{n_{m+1}}$ is nonexpansive, we have
	\begin{equation*}
	d(y^m_{n_{m+1}},y^{m-1}_{n_{m+1}})\leq (1-\alpha_{n_{m+1}})d(y^m_{n_{m+1}},y^{m-1}_{n_{m+1}})+\alpha_{n_{m+1}}d(\phi(z_m),\phi(z_{m-1})),
	\end{equation*}
	which gives $d(y^m_{n_{m+1}},y^{m-1}_{n_{m+1}})\leq d(\phi(z_m),\phi(z_{m-1}))$.
	Thus, for $m\in [1, M]$,
	\begin{equation*}
	d(z_{m+1},z_m)\leq d(y^m_{n_{m+1}},y^{m-1}_{n_{m+1}})+d(y^{m-1}_{n_{m+1}},y^{m-1}_{n_m})\leq d(\phi(z_m),\phi(z_{m-1}))+\varepsilon_0.
	\end{equation*}
	
	Now we argue that there is $m_0\leq M$ such that $d(z_{m_0+1},z_{m_0})\leq \tilde{\varepsilon}$. Assume that for all $m\leq M-1$, $d(z_{m+1},z_m)>\tilde{\varepsilon}$, which gives $d(\phi(z_{m+1}),\phi(z_m))\leq (1-\delta(\tilde{\varepsilon}))d(z_{m+1},z_m)$, since $\delta$ is a Rakotch-modulus for $\phi$. Hence, for $m\in[1,M]$
	\begin{equation*}
	d(z_{m+1},z_{m})\leq (1-\delta(\tilde{\varepsilon}))d(z_m,z_{m-1})+\varepsilon_0.
	\end{equation*}
	By induction
	\begin{align*}
	d(z_{M+1},z_{M})&\leq (1-\delta(\tilde{\varepsilon}))^Md(z_1,z_0)+\varepsilon_0\sum_{i=0}^{M-1}\left(1-\delta(\tilde{\varepsilon})\right)^i\\
	&\leq \frac{\tilde{\varepsilon}}{2b}b + \frac{\varepsilon_0}{\delta(\tilde{\varepsilon})}=\tilde{\varepsilon}.
	\end{align*}
	
	Let $m_0\leq M$ be such that $d(z_{m_0+1},z_{m_0})\leq \tilde{\varepsilon}$. We now write $y_n\equiv y^{m_0}_n$ and $z=z_{{m_0+1}}$. We have $[n_{m_0+1}, f(n_{m_0+1})]\subset [n_{m_0+1}, f_{M-m_0}(n_{m_0+1})]$ and so
	\begin{equation}\label{distanceyitoz}
	\forall i \in[n_{m_0+1}, f(n_{m_0+1})]\, \left(d(y_i,z)\leq \varepsilon_0\leq \tilde{\varepsilon}\right).
	\end{equation}
	By (W4) we have for $i\in[n_{m_0+1}, f(n_{m_0+1})]$ (using again that 
$S_i$ is nonexpansive),
	\begin{equation*}
	d(x_i,y_i)\leq (1-\alpha_i)d(x_i,y_i)+\alpha_id(\phi(x_i),\phi(z_{m_0})),
	\end{equation*}
	which implies, using the fact that $\phi$ is nonexpansive,
	\begin{align}
	d(x_i,y_i)&\leq d(\phi(x_i),\phi(z_{m_0})) \leq d(\phi(x_i),\phi(z))+d(\phi(z),\phi(z_{m_0}))\nonumber\\
	&\leq d(\phi(x_i),\phi(z))+d(z_{m_0+1},z_{m_0})\leq d(\phi(x_i),\phi(z))+\tilde{\varepsilon}\label{distancexitoyi}.
	\end{align}
	An inductive argument and the monotoniticty of $\theta_b$ is the last variable, entail $N\le n_{m_0+1}\leq n_{M+1}\leq \Psi(\varepsilon,f, N)$. We now argue that for $i\in[n_{m_0+1}, f(n_{m_0+1})]$, $d(x_i,z)\leq \frac{\varepsilon}{2}$. Assume that for some $i\in[n_{m_0+1}, f(n_{m_0+1})]$, $d(x_i,z)>\frac{\varepsilon}{2}$. Since $\delta$ is a Rakotch-modulus for $\phi$,
	\begin{equation*}
	d(\phi(x_i),\phi(z))\leq (1-\delta(\varepsilon/2))d(x_i,z),
	\end{equation*}
	and then, using \eqref{distanceyitoz} and \eqref{distancexitoyi},
	\begin{align*}
	d(x_i,z)&\leq d(x_i,y_i)+d(y_i,z)\leq d(\phi(x_i),\phi(z))+ \tilde{\varepsilon} + d(y_i,z)\\
	&\leq (1-\delta(\varepsilon/2))d(x_i,z)+2\tilde{\varepsilon}. 
	\end{align*}
	This implies the contradiction $d(x_i,z)\leq \frac{2\tilde{\varepsilon}}{\delta(\varepsilon/2)}=\frac{\varepsilon}{2}$. \\ 
Using $d(z,z_{m_0})\le \tilde{\varepsilon}\le \varepsilon,$ (W4) and the 
nonexpansivity of $S_{n_{m_0+1}}$ and $\phi$, we get  
\[ d(z,y_{n_{m_0+1}}(\phi(z)))=
d(y_{n_{m_0+1}}(\phi(z_{m_0})),y_{n_{m_0+1}}(\phi(z)))\le \varepsilon,\]
which entails that the theorem is satisfied with $z$ and $n:=n_{m_0+1}$.
\end{proof}

\begin{remark}\label{finitization1}
{\rm Theorem \ref{SBtheo}} implies that (already for $N=0$) 
$\Psi(\varepsilon,f,N)$ is a rate of metastability for $(x_n).$ 
We wrote the statement in a more informative form 
so that it implies in an elementary way that (if $X$ is 
complete and $C$ is closed) the limit $x$ of $(x_n)$ actually is the (unique) 
fixed point of $P\circ\phi,$ where ${P(u):=\lim y_n(u)}$, which is 
an important further information in Suzuki's theorem: let $\varepsilon >0$ and 
$N\in\N$ be so large that 
\[ (1)\ \forall k\ge N\ (d(x_k,x)\,,\, 
d(y_k(\phi(x)),P(\phi(x)))\le \varepsilon). \] 
Then by the theorem, we get an $n\ge N$ and a $z\in C$ such that 
\[ (2) \ d(z,y_n(\phi(z)))\,,\, d(x_n,z)\le \varepsilon.\]
Using $(1),(2)$  $d(x,z)\le 2\varepsilon$ and so by (W4) and the nonexpansivity 
of $S_n$ and $\phi$ 
\[ (3) \ d(y_n(\phi(x)),y_n(\phi(z)))\le 2\varepsilon. \]
Hence using $(2)$,  
\[ (4) \ d(x,y_n(\phi(x)))\le d(z,y_n(\phi(z)))+4\varepsilon\le 5\varepsilon\]
and so by $(1)$ 
\[ (5) \ d(x,P(\phi(x)))\le 6\varepsilon. \]
Since $\varepsilon >0$ was arbitrary, we have that $x=P(\phi(x)).$
\end{remark}

\begin{remark}\label{unboundedness}
	{\rm Theorem~\ref{SBtheo}} can be adapted to dispense with the boundedness assumption on the set $C$, and assume the existence of a natural number $b\in\N^*$ that only satisfies $b\geq \max\{d(x_0, \phi(x_0)), d(x_0, y_n(\phi(x_0))) : n \in \N\}$. In the original proof by Suzuki, one considers that all the $(S_n)$-Browder type sequences converge and thus, the particular fact that $(y_n(\phi(x_0)))$ is bounded trivially follows. In this situation, one assumes that for each $r\in \N^*$, the function $\theta_{r}$ satisfies the formula \eqref{BSncondition} only for anchor points $u\in C \cap B_r(x_0)$, where $B_r(x_0)$ denotes the closed ball centred at $x_0$ with radius $r$. Then, {\rm Theorem~\ref{SBtheo}} can be adapted to hold with the bound $\Psi(\varepsilon, f,N)=\theta_{(M+1)b}\left(\varepsilon_0, f, \Psi_M\right)$, where $\tilde{\varepsilon}$, $\varepsilon_0$, $M$ are as before and
	\[
	\Psi_0=N \text{ and } \Psi_{m+1}=\theta_{(m+1)b}\left(\varepsilon_0, f_{M-m},
	\Psi_m\right),
	\]
	with, for all $p\in\N$
	\begin{equation*}
			\begin{cases}
				f_0(p)=f(p),\\
				f_{m+1}(p)=\max\left\{f(p),\, \theta_{(M-m+1)b}\left(\varepsilon_0, f_m, p\right)\right\} & \text{for } m<M.
			\end{cases}
	\end{equation*}
	Indeed, the proof can be carried out with similar arguments. The definition of the natural numbers $n_m$ and the points $z_m$ is similar to before. Assume $n_m$ and $z_m$ are already defined, with $d(x_0, z_m)\leq m\cdot b$. From the hypothesis on $\theta_{(\cdot)}$ and the fact that
	\[
	d(x_0, \phi(z_m))=d(x_0, \phi(x_0)) + d(\phi(x_0), \phi(z_m))\leq (m+1)b,
	\]
	we can define $n_{m+1}\in [n_m, \theta_{(m+1)b}(\eps_0, f_{M-m}, n_m)]$ satisfying
	\[
	\forall i\in[n_{m+1}, f_{M-m}(n_{m+1})]\, \left( d(y^m_i, y^m_{n_{m+1}})\leq \eps_0 \right).
	\]
	Moreover, with $z_{m+1}=y^m_{n_{m+1}}$, we have
	\[
	\begin{split}
	d(x_0, z_{m+1})&\leq d(x_0, y_{n_{m+1}}(\phi(x_0))) + d(y_{n_{m+1}}(\phi(x_0)), y_{n_{m+1}}(z_m))\\
	&\leq b+ d(x_0, z_m) \leq (m+1)b,	
	\end{split}
	\]
	which ensures that this construction holds for all $m\leq M$. Since it still holds that $[n_{m+1}, \theta_{(m+2)b}(\eps_0, f_{M-m-1}, n_{m+1})]\subset [n_{m+1}, f_{M-m}(n_{m+1})]$, we again have $d(y^m_{n_{m+2}}, y^m_{n_{m+1}})\leq \eps_0$.
	The remainder of the proof follows the same arguments as before, noticing that the new condition on the natural number $b$ suffices to show the existence of $z_{m_0}$ satisfying $d(z_{m_0+1}, z_{m_0})\leq \tilde{\varepsilon}$, as it still entails ${(1-\delta(\tilde{\varepsilon}))^M}d(z_1, z_0)\leq\tilde{\varepsilon}/2$.
\end{remark}

An application of Theorem~\ref{SBtheo} gives a rate of metastability for the Browder-viscosity iteration from a function $\theta_b$ satisfying \eqref{Bcondition}.
\begin{corollary}\label{VBcor1}
		Let $C$ be a nonempty bounded convex subset of a hyperbolic space and let $b\in \N^*$ be a bound on the diameter of $C$. Consider $T$ a nonexpansive map on $C$, an $r$-contraction $\phi:C\to C$  with $\delta\in(0,1)$ such that $r\leq 1-\delta$, and a sequence $(\alpha_n)\subset (0,1]$. Let $\theta_b$ be a monotone function satisfying \eqref{Bcondition}. Then $(x_n)$ satisfying \begin{equation*}
		x_n=(1-\alpha_n)T(x_n)\oplus \alpha_n\phi(x_n), \, n\in\N
		\end{equation*}
		is a Cauchy sequence with metastability rate
		\begin{equation*}
		\Psi(\varepsilon, f)=\Psi[b, \delta, \theta_b](\varepsilon, f)=\theta_b\left(\varepsilon_0, f, \Psi_M\right),
		\end{equation*}
		where
	\begin{equation*}
	\begin{gathered}
	\Psi_0=0 \text{ and } \Psi_{m+1}=\theta_b\left(\varepsilon_0, f_{M-m}, \Psi_m\right), \text{ with }\\
	\varepsilon_0=\frac{\varepsilon\delta^2}{8}, \, M=\left\lceil\log_{1-\delta}\left(\frac{\varepsilon\delta}{8b}\right)\right\rceil \text{ and}\\
	\text{for all } p\in\N,\, \begin{cases}
	f_0(p)=f(p),\\
	f_{m+1}(p)=\max\left\{f(p),\, \theta_b\left(\varepsilon_0, f_m, p\right)\right\} & \text{for } m<M.
	\end{cases}
	\end{gathered}
	\end{equation*}
\end{corollary}

\begin{corollary} \label{VB-rate}
Let $X,C,T,\phi,b,r,\delta,(\alpha_n),(x_n)$ be as in Corollary 
\ref{VBcor1}. Let $\rho$ be a common Cauchy rate for the $T$-Browder 
iteration $(y_n(u))$ for all $u\in C$ used as anchor points. 
Then 
\[ \Psi[b,d,\rho](\varepsilon):=
\rho\left(\frac{\varepsilon\delta^2}{8}\right) \] 
is a Cauchy rate for $(x_n).$
\end{corollary}
\begin{proof}
By Proposition \ref{prop-rate}, $\theta_b(\varepsilon,f,N):=\rho'(\varepsilon,N)
:=\max\{ N,\rho(\varepsilon)\}$ 
is a monotone function satisfying \eqref{Bcondition}. 
Hence by Corollary \ref{VBcor1}
\[ \Psi[b,d,\rho](\varepsilon):=\rho'(\varepsilon_0,\Psi_M),\] 
where 
\[ \Psi_0=0 \text{ and } \Psi_{m+1}=\rho'\left(\varepsilon_0, \Psi_m\right), \,
\mbox{with}\ 
	\varepsilon_0=\frac{\varepsilon\delta^2}{8}, \, M=\lceil\log_{1-\delta}\left(\frac{\varepsilon\delta}{8b}\right)\rceil, \]
is a rate of metastability for $(x_n)$ which does not depend on $f$ and 
hence - again by Proposition \ref{prop-rate} - is a rate of convergence. An easy induction shows 
that $\Psi_m\le \rho(\eps_0)$ for all $m$ and so 
$\rho'(\eps_0,\Psi_M)=\rho(\eps_0).$
\end{proof}

The Browder type sequence has been the focus of many quantitative studies using techniques from proof mining. In \cite{Kohlenbach(11)}, the first author analysed Browder's fixed point theorem in Hilbert spaces showing that it was possible to avoid sequential weak compactness used in the original proof by Browder. A simpler proof of Browder's result due to Halpern in \cite{Halpern(67)} that already avoids weak compactness, was also analyzed and a metastability bound was obtained. Kirk~\cite{Kirk(03)} showed that Halpern's proof can be extended to $\CAT$ spaces (essentially unchanged), and thus the same holds true for its quantitative version \cite[Prop.9.3]{KohlenbachLeustean(12)}. By Theorem~\ref{SBtheo}, using the rates of metastability extracted in \cite{KohlenbachLeustean(12)}, we obtain rates of metastability for the viscosity-Browder type sequences in the setting of $\CAT$ spaces. In \cite{Koernlein(diss)} and \cite[Theorem 6.7]{Koernlein(descent)}, for the special case where $X$ is a Hilbert space, similar rates of metastability have already been extracted - again using methods of proof mining - from a direct proof of the strong convergence of $(x_n)$ 
(i.e. without a reduction to Browder's theorem) from \cite{Yamada(01)}.

In \cite{KohlenbachSipos(ta)}, a quantitative treatment of Reich's extension of Browder fixed point theorem to uniformly smooth Banach spaces was carried out under the additional hypothesis that the space is uniformly convex. Using the rate of metastability extracted in \cite{KohlenbachSipos(ta)}, once again by an application of Theorem~\ref{SBtheo}, we obtain a rate of metastability for the viscosity-Browder type sequence in Banach spaces that are both uniformly convex and uniformly smooth.

\subsection{Halpern-viscosity}
In this section we establish a quantitative version of Theorem 8 in 
\cite{Suzuki(07)}.

\begin{theorem}\label{SHtheo}
	Let $C$ be a nonempty bounded convex subset of a hyperbolic space and let $b\in \N^*$ be a bound on the diameter of $C$. Consider $(S_n)$ a family of nonexpansive maps on $C$, a MKC $\phi:C\to C$  with Rakotch-modulus $\delta$, and a sequence $(\alpha_n)\subset [0,1]$ satisfying $\sum \alpha_n=\infty$ with a monotone rate of divergence $A$. Let $\theta'_b$ be a monotone function satisfying \eqref{HSncondition} and $\sigma_1$ be as in Lemma \ref{Hlem1}. 
Define 
	\begin{equation*}
	\Psi(\varepsilon, f,N)=\Psi[b, \delta, A, \theta'_b](\varepsilon, f,N)=\sigma_1[\widetilde{A}_{\varepsilon}, b]\left(\frac{\varepsilon}{3}, \Psi_{M+1}\right),
	\end{equation*}
	where
	\begin{equation*}
	\begin{gathered}
	\widetilde{A}_{\varepsilon}(k):=A\left(\left\lceil \frac{k}{\delta(\varepsilon/3)}\right\rceil\right) \text{ for all } k\in \N\\
	\Psi_0=\max\{ N,A(1)+1\} \text{ and } \Psi_{m+1}=\theta_b'\left(\varepsilon_0, f_{M-m}, \Psi_m\right), \text{ with }\\
	\tilde{\varepsilon}=\frac{2\varepsilon\delta(\varepsilon/3)}{15}, \, \varepsilon_0=\frac{\tilde{\varepsilon}\delta(\tilde{\varepsilon})}{4}, \, M=\left\lceil\log_{1-\frac{\delta(\tilde{\varepsilon})}{2}}\left(\frac{\tilde{\varepsilon}}{2b}\right)\right\rceil\, \text{ and}\\
	\text{for all } p\in\N,\, \begin{cases}
	f_0(p)=\max\left\{ f\left(\sigma_1[\widetilde{A}_{\varepsilon}, b](\frac{\varepsilon}{3}, p)\right),\sigma_1[\widetilde{A}_{\varepsilon}, b](\frac{\varepsilon}{3}, p)\right\},\\
	f_{m+1}(p)=\max\{f_0(p),\,\theta_b'(\varepsilon_0, f_m, p)\} & \text{for } m<M.
	\end{cases}
	\end{gathered}
	\end{equation*} 
Then for $(x_n)$ satisfying \eqref{VSnH} we have 
\[ \begin{array}{l} \forall N\in\N \,\forall \varepsilon >0\, \forall f:\N\to\N\,
\exists n\in [N,\Psi(\varepsilon, f,N)]\,\exists k\in [N,n], z\in C\, 
\\[1mm] \hspace*{5mm} \left( d(z,w_k(\phi(z)))  \le \varepsilon \wedge 
\forall i\in [n,f(n)]\, (d(x_i,z)\le \frac{\varepsilon}{2})\right). \end{array} \]
In particular, $\Psi(\varepsilon, f,0)$ is a rate of metastability for $(x_n).$
\end{theorem}

\begin{proof}
	Let $N\in \N$, $\varepsilon >0$ and $f:\N\to\N$ be given. From the definition of the natural number $M$, we have $\left(1-\dfrac{\delta(\tilde{\varepsilon})}{2}\right)^M\leq \dfrac{\tilde{\varepsilon}}{2b}$. Consider the functions $f_0, \cdots, f_M$ as in the theorem.
	Define $z_0=x_0$ and $n_0=\max\{ N,A(1)+1\}$. Notice that $$\prod_{i=0}^{n_0-1}(1-\alpha_i)\leq \exp(-\sum_{i=0}^{n_0-1}\alpha_i)\leq \frac{1}{2}.$$
	For all $m\leq M$, assume that $z_m$ and $n_m$ are already defined and let $(w_n^m)_n$ denote the $(S_n)$-Halpern sequence with anchor and starting point $\phi(z_m)$, i.e. $w_n^m=w_n(\phi(z_m))$. By \eqref{HSncondition}, we consider $n_{m+1}\in[n_m, \theta_b'(\varepsilon_0, f_{M-m}, n_m)]$ such that
	\[\forall i\in[n_{m+1}, f_{M-m}(n_{m+1})]\, \left( d(w_i^m,w^m_{n_{m+1}})\leq \varepsilon_0\right),\]
	and define $z_{m+1}=w^m_{n_{m+1}}$.
	
	Notice that for all $m\leq M$, since $n_m\geq n_0\geq N$, we have $$\prod_{i=0}^{n_m-1}(1-\alpha_i)\leq \frac{1}{2}.$$
	Moreover, for $m\leq M-1$, $[n_{m+1}, \theta_b'(\varepsilon_0, f_{M-m-1}, n_{m+1})] \subseteq [n_{m+1}, f_{M-m}(n_{m+1})]$ and so $d(w^m_{n_{m+2}},z_{m+1})=d(w^m_{n_{m+2}},w^m_{n_{m+1}})\leq \varepsilon_0$.
	
	We start by showing that there is $m_0\leq M$ such that $d(z_{m_0+1},z_{m_0})\leq \tilde{\varepsilon}$. Assume that for all $m\leq M-1$, $d(z_{m+1},z_m)>\tilde{\varepsilon}$. For $m\in[1, M]$,
	\begin{equation*}
	d(z_{m+1},z_m)\leq d(w^m_{n_{m+1}},w^{m-1}_{n_{m+1}})+d(w^{m-1}_{n_{m+1}},w^{m-1}_{n_m})\leq d(w^m_{n_{m+1}},w^{m-1}_{n_{m+1}})+\varepsilon_0,
	\end{equation*}
	and by induction on $n$, we have (using $(W4)$ and the nonexpansivity of $S_n$) 
	\begin{equation*}
	d(w^m_n,w^{m-1}_n)\leq \left((1-\delta(\tilde{\varepsilon}))+\delta(\tilde{\varepsilon})\prod_{i=0}^{n-1}(1-\alpha_i)\right)d(z_m,z_{m-1}).
	\end{equation*}
	Then for $m\in[1, M]$, $d(w^m_{n_{m+1}},w^{m-1}_{n_{m+1}})\leq \left(1-\frac{\delta(\tilde{\varepsilon})}{2}\right)d(z_m,z_{m-1})$. Hence,
	\begin{align*}
	d(z_{m+1},z_m)& \leq \left(1-\frac{\delta(\tilde{\varepsilon})}{2}\right)d(z_m,z_{m-1}) +\varepsilon_0\\
	&\leq \left(1-\frac{\delta(\tilde{\varepsilon})}{2}\right)^m d (z_1,z_0) + \varepsilon_0\sum_{i=0}^{m-1} \left(1-\frac{\delta(\tilde{\varepsilon})}{2}\right)^i.
	\end{align*}
	We conclude,
	\begin{equation*}
	d(z_{M+1},z_M)\leq \left(1-\frac{\delta(\tilde{\varepsilon})}{2}\right)^Mb+\frac{2\varepsilon_0}{\delta(\tilde{\varepsilon})}\leq \frac{\tilde{\varepsilon}}{2b}b+\frac{\tilde{\varepsilon}}{2}=\tilde{\varepsilon}.
	\end{equation*}
	Thus, we can take $m_0\leq M$ such that $d(z_{m_0+1},z_{m_0})\leq \tilde{\varepsilon}$. Write $w_n\equiv w^{m_0}_n$, $z=z_{m_0+1}=w^{m_0}_{n_{m_0+1}}$ and $\sigma_1=\sigma_1[\widetilde{A}_{\varepsilon}, b](\frac{\varepsilon}{3}, n_{m_0+1})\ge 
n_{m_0+1}$. Assume towards a contradiction that
	\begin{equation*}
	\forall i\in[n_{m_0+1}, \sigma_1]\, \left(d(x_i,w_i) >\frac{\varepsilon}{3}\right).
	\end{equation*}
	For all $i\in [n_{m_0+1}, \sigma_1]$, by (W4)
	\begin{align*}
	d(&x_{i+1},w_{i+1})\leq  (1-\alpha_i)d(S_i(x_i),S_i(w_i))+\alpha_id(\phi(x_i),\phi(z_{m_0}))\\
	&\leq (1-\alpha_i)d(x_i,w_i)+\alpha_i\left(d(\phi(x_i),\phi(w_i))+d(\phi(w_i),\phi(z))+d(\phi(z),\phi(z_{m_0}))\right)\\
	&\leq \left(1-\delta\left(\varepsilon/3\right)\alpha_i\right)d(x_i,w_i)+\delta\left(\varepsilon/3\right)\alpha_i\frac{d(w_i,z)+d(z,z_{m_0})}{\delta(\varepsilon/3)}\\
	&\leq (1-\widetilde{\alpha}_i)d(x_i,w_i)+\tilde{\alpha}_ib_i,
	\end{align*}
	where $\tilde{\alpha}_i=\delta(\varepsilon/3)\alpha_i$ and $b_i=\dfrac{d(w_i,z)+d(z,z_{m_0})}{\delta(\varepsilon/3)}$.
	
	Using the monotonicity of the function $\theta_b'$ it is easy to see that $n_{m}\leq \Psi_{m}$, for all $m\leq M+1$. Since $A$ is a rate of divergence for $\sum \alpha_n=\infty$, we have that $\widetilde{A}_{\varepsilon}$ is a rate of divergence for $\sum \delta(\varepsilon/3)\alpha_n=\infty$. Furthermore, by the monotonicity of the function $A$ it follows that for all $\varepsilon >0$, the function $\sigma_1[\widetilde{A}_{\varepsilon}, b]$ is also monotone. Thus from $n_{m_0+1}\leq n_{M+1}$, we have $\sigma_1\leq \Psi(\varepsilon,f,N)$. By the definition of the functions $f_m$, it follows $\sigma_1\leq f_{M-m_0}(n_{m_0+1})$. Hence
	\begin{equation*}
	\forall i\in[n_{m_0+1}, \sigma_1]\, \left(d(w_i,z)\leq \varepsilon_0\leq \frac{\tilde{\varepsilon}}{4}\right),
	\end{equation*}
	and so,
	\begin{equation*}
	b_i\leq \frac{\tilde{\varepsilon}/4+\tilde{\varepsilon}}{\delta(\varepsilon/3)}= \frac{\varepsilon}{6}=\frac{(\varepsilon/3)}{2}.
	\end{equation*}
	By Lemma~\ref{Hlem1_restricted}, we have that $d(x_{\sigma_1},w_{\sigma_1})\leq \frac{\varepsilon}{3}$, which contradicts the assumption. Consider $i_0$ such that
	\begin{equation*}
	i_0\in[n_{m_0+1}, \sigma_1] \text{ and } d(x_{i_0},w_{i_0})\leq \frac{\varepsilon}{3}.
	\end{equation*}
	Note that for all $u, v\in C$, it holds $d(\phi(u),\phi(v))\leq \max\{\left(1-\delta(\varepsilon/3)\right)d(u,v), \varepsilon/3\}$.
	For any $i\in [i_0, f_{M-m_0}(n_{m_0+1})-1]$ such that $d(x_i,w_i)\leq \frac{\varepsilon}{3}$, by (W4) and the nonexpansivity of $S_i$, we have
	\begin{align*}
	&d(x_{i+1},w_{i+1})\leq (1-\alpha_i)d(x_i,w_i)+\alpha_id(\phi(x_i),\phi(z_{m_0}))\\
	&\leq (1-\alpha_i)d(x_i,w_i)+\alpha_i\max\left\{\left(1-\delta\left(\frac{\varepsilon}{3}\right)\right)d(x_i,z_{m_0}),\, \frac{\varepsilon}{3})\right\}\\
	&= \max\left\{(1-\alpha_i)d(x_i,w_i)+\alpha_i\left(1-\delta\left(\frac{\varepsilon}{3}\right)\right)d(x_i,z_{m_0}),\, (1-\alpha_i)d(x_i,w_i)+\alpha_i\frac{\varepsilon}{3}\right\}\\
	&\leq \max\left\{\left(1-\delta\left(\frac{\varepsilon}{3}\right)\alpha_i\right)d(x_i,w_i)+\alpha_i\delta\left(\frac{\varepsilon}{3}\right)\left(1-\delta\left(\frac{\varepsilon}{3}\right)\right)\frac{d(w_i,z)+d(z,z_{m_0})}{\delta(\frac{\varepsilon}{3})},\right.\\
	&\left.\qquad\qquad(1-\alpha_i)d(x_i,w_i)+\alpha_i\frac{\varepsilon}{3}\right\}\\
	&\leq \max\{(1-\tilde{\alpha}_i)\frac{\varepsilon}{3} + \tilde{\alpha}_i\frac{\varepsilon}{3},\, (1-\alpha_i)\frac{\varepsilon}{3}+\alpha_i\frac{\varepsilon}{3}\}=\frac{\varepsilon}{3},
	\end{align*}
	using the inequalities $d(x_i,w_i)\leq \dfrac{\varepsilon}{3}$, $1-\delta(\varepsilon/3)\leq 1$ and $$\frac{d(w_i,z)+d(z,z_{m_0})}{\delta(\varepsilon/3)}\leq \dfrac{\varepsilon_0+\tilde{\varepsilon}}{\delta(\varepsilon/3)}\leq \dfrac{\varepsilon}{3}.$$
	
	Hence, $d(x_i,w_i)\leq \dfrac{\varepsilon}{3}$ for all $i\in [i_0, f_{M-m_0}(n_{m_0+1})]$ and in particular, the inequality holds for $i\in [\sigma_1, f(\sigma_1)]$. The theorem is now satisfied with $n:=\sigma_1$, $k:=n_{m_0+1}$ and $z$ as defined above. Indeed, since $d(z,z_{m_0})=d(z_{m_0+1},z_{m_0})\le \varepsilon$ and $[n, f(n)]\subseteq [n_{m_0+1}, f_{M-m_0}(n_{m_0+1})]$, using (W4) and the nonexpansivity of $S_{n_{m_0+1}}$ and $\phi$, we have 
\[ d(z,w_k(\phi(z)))=d(w_{n_{m_0+1}}(\phi(z_{m_0})),w_{n_{m_0+1}}(\phi(z_{m_0+1})))\le \varepsilon \]
and for all $i\in [n,f(n)]$
\[ d(x_i,z)\le d(x_i,w_i)+d(w_i,z)\le 
\frac{\varepsilon}{3}+\varepsilon_0\le \frac{\varepsilon}{2}. \qedhere\]
\end{proof}

\begin{remark} 
As in {\rm Remark~\ref{finitization1}}, it follows from 
{\rm Theorem \ref{SHtheo}} in an elementary way that (for $X$ being complete and $C$ being 
closed) the limit $x:=\lim x_n$ is the unique fixed point of $P\circ\phi,$ 
where $P(u):=\lim w_n(u).$
\end{remark} 

\begin{remark}\label{unboundedness2}
	As in {\rm Remark~\ref{unboundedness}}, instead of assuming that $C$ is bounded and working with a bound on its diameter, we can adapt {\rm Theorem~\ref{SHtheo}} to the situation where we only have $b\in\N^*$ satisfying $b\geq \max\{d(x_0, w_n(\phi(x_0))) : n\in \N\}$. We also require bounding information on the initial displacement of the relevant mappings w.r.t. $x_0$, i.e. for each $n\in\N$ let $c_n\in\N$ be such that
	\[
	c_n\geq \max\{ d(x_0, \phi(x_0)), d(x_0, S_n(x_0))\}.
	\]
	In this case, the bound takes the form $\Psi(\eps, f, N)=\sigma_1[\widetilde{A}_{\eps}, b']\left( \frac{\eps}{3}, \Psi_{M+1} \right)$, where $\tilde{\varepsilon}, \eps_0, \widetilde{A}_{\eps}$ and $M$ are as before,
	\[
	\Psi_0=\max\{N, A(1)+1\} \text{ and } \Psi_{m+1}=\theta'_{(m+1)b}\left(\varepsilon_0, f_{M-m},
	\Psi_m\right),
	\]
	$$b'=(M+1)\cdot b+\sum_{k=0}^{\Psi_{M+1}-1}c_k,$$ and for all $p\in\N$,
	\begin{equation*}
		\begin{cases}
			f_0(p)=\max\{f\left(\sigma_1[\widetilde{A}_{\eps}, b'](\frac{\eps}{3},p)\right), \sigma_1[\widetilde{A}_{\eps}, b'](\frac{\eps}{3},p)\},\\
			f_{m+1}(p)=\max\left\{f_0(p),\, \theta'_{(M-m+1)b}\left(\varepsilon_0, f_m, p\right)\right\} & \text{for } m<M.
		\end{cases}
	\end{equation*}
	We can define $n_m$ and $z_m$ in a way similar to {\rm Remark~\ref{unboundedness}}. Since it still holds that $\left(1-\delta(\tilde{\varepsilon})/2\right)^Md(z_1, z_0)\leq \tilde{\varepsilon}/2$, we can conclude the existence of $z_{m_0}$ satisfying $d(z_{m_0+1},z_{m_0})\leq \tilde{\varepsilon}$. The next argument requires an application of {\rm Lemma~\ref{Hlem1_restricted}}. Let $w_n\equiv w_n^{m_0}$ be as before, write $\sigma_1=\sigma_1[\tilde{A}_{\eps}, b'](\frac{\eps}{3}, n_{m_0+1})$ and assume that 
	\[
	\forall i\in [n_{m_0+1}, \sigma_1]\, \left(d(x_i, w_i) > \frac{\eps}{3}\right).
	\]
	To apply {\rm Lemma~\ref{Hlem1_restricted}}, it is enough to see that for all $i\in [n_{m_0+1}, \sigma_1+1]$, we have $d(x_i,w_i)\leq b'$. First, inductively it holds $d(x_i,w_i)\leq d(x_{n_{m_0+1}},w_{n_{m_0+1}})$, for $i\in [n_{m_0+1}, \sigma_1+1]$. The base case is trivial. In the induction step we argue as before: For $i\in [n_{m_0+1}, \sigma_1]$, we have
	\[
	\begin{split}
	d(x_{i+1},w_{i+1})&\leq (1-\widetilde{\alpha}_i)d(x_i,w_i)+ \widetilde{\alpha}_i b_i\\
	&\leq (1-\widetilde{\alpha}_i)d(x_i,w_i)+ \widetilde{\alpha}_id(x_{n_{m_0+1}},w_{n_{m_0+1}})\\
	&\leq d(x_{n_{m_0+1}},w_{n_{m_0+1}}),
	\end{split}
	\]
	using the induction hypothesis and $b_i\leq \eps/6< \eps/3<d(x_{n_{m_0+1}},w_{n_{m_0+1}})$.
	Now, by induction, one shows $d(x_0, x_n)\leq \sum_{k=0}^{n-1}c_k$, for all $n\in\N$. Then, from $n_{m_0+1}\leq n_{M+1}\leq \Psi_{M+1}$ and $d(x_0,z_{m_0+1})\leq (m_0+1)b\leq (M+1)b$, we conclude
	\begin{equation*}
	d(x_{n_{m_0+1}},w_{n_{m_0+1}})\leq d(x_0, x_{n_{m_0+1}}) + d(x_0, z_{m_0+1}) \leq b'.
	\end{equation*}
	The remainder of the proof follows unchanged.
	\end{remark}

An application of Theorem~\ref{SHtheo} gives a rate of metastability for the Halpern-viscosity iteration from a function $\theta'_b$ satisfying \eqref{Hcondition}.
\begin{corollary}\label{VHcor1}
	Let $C$ be a nonempty bounded convex subset of a hyperbolic space and let $b\in \N^*$ be a bound on the diameter of $C$. Consider $T$ a nonexpansive map on $C$, an $r$-contraction $\phi:C\to C$  with  $\delta\in(0,1)$ such that $r\leq 1-\delta$, and a sequence $(\alpha_n)\subset [0,1]$ satisfying $\sum \alpha_n=\infty$ with a monotone rate of divergence $A$. Let $\theta'_b$ be a monotone function satisfying \eqref{Hcondition} and $\sigma_1$ be as in Lemma \ref{Hlem1}. 
Then $(x_n)$ defined by
	\begin{equation*}
	x_0\in C, \qquad x_{n+1}=(1-\alpha_n)T(x_n)\oplus \alpha_n\phi(x_n), \, n\in\N,
	\end{equation*}
	 is a Cauchy sequence with metastability rate
	 \begin{equation*}
	 \Psi(\varepsilon, f)=\Psi[b, \delta, A, \theta'_b](\varepsilon, f)=\sigma_1[\widetilde{A}, b]\left(\frac{\varepsilon}{3}, \Psi_{M+1}\right)
	 \end{equation*}
	 where
	\begin{equation*}
	\begin{gathered}
	\widetilde{A}(k):=A\left(\left\lceil \frac{k}{\delta}\right\rceil\right) \text{ for all } k\in \N\\
	\Psi_0=A(1)+1 \text{ and } \Psi_{m+1}=\theta_b'\left(\varepsilon_0, f_{M-m}, \Psi_m\right), \text{ with }\\
	\varepsilon_0=\frac{\varepsilon\delta^2}{30}, \, M=\left\lceil\log_{1-\frac{\delta}{2}}\left(\frac{\varepsilon\delta}{15b}\right)\right\rceil\, \text{ and}\\
	\text{for all } p\in\N,\, \begin{cases}
	f_0(p)=f\left(\sigma_1[\widetilde{A}, b](\frac{\varepsilon}{3}, p)\right),\\
	f_{m+1}(p)=\max\{f_0(p),\,\theta_b'(\varepsilon_0, f_m, p)\} & \text{for } m<M.
	\end{cases}
	\end{gathered}
	\end{equation*}
\end{corollary}
Just as in the case of Corollary \ref{VB-rate} we obtain the following 
consequence of Corollary \ref{VHcor1}:
\begin{corollary} 
Let $X,C,b,T,\phi,r,\delta,(\alpha_n),A,(x_n), \sigma_1$ be as in 
 Corollary \ref{VHcor1}. Let $\rho(\varepsilon)$ be a common (for all 
anchor points $u\in C$) rate of convergence for the sequence $(w_n(u))$ 
of Halpern iterations of $T.$ Then
	 \begin{equation*}
	 \Psi(\varepsilon, f)=\Psi[b, \delta, A, \rho](\varepsilon, f)=\sigma_1[\widetilde{A}, b]\left(\frac{\varepsilon}{3}, \max\{ \rho(\eps_0),A(1)+1\}\right)
	 \end{equation*}
	 where
	\begin{equation*}
	\widetilde{A}(k):=A\left(\left\lceil \frac{k}{\delta}\right\rceil\right) \text{ for all } k\in \N \text{ and } 
\varepsilon_0=\frac{\varepsilon\delta^2}{30}. 
\end{equation*}
is a Cauchy rate for $(x_n).$
\end{corollary}

Rates of metastability for the Halpern iterations have been extracted before. In the setting of $\CAT$ spaces, under appropriate conditions, rates of metastability were obtained in \cite{KohlenbachLeustean(12)}. The convergence of Halpern iterations was also studied by proof mining methods in the setting of Banach spaces. The results \cite{KohlenbachLeustean(12)i} in and \cite{Koernlein(15)} (under different sets of conditions, both including the natural choice $\alpha_n=\frac{1}{n+1}$) extracted a transformation of a rate of metastability for a certain Browder type sequence into a rate of metastability for the Halpern iteration. Thus, the recent rate obtained in \cite{KohlenbachSipos(ta)} entails a rates of metastability for Halpern iterations in the setting of Banach spaces that are simultaneously uniformly convex and uniformly smooth. Subsequently, by an application of Theorem~\ref{SHtheo}, instantiated with the corresponding rates of metastability for Halpern iterations, we obtain rates of metastability for the viscosity-Halpern iterations in the setting of $\CAT$ spaces and in the setting of Banach spaces that are both uniformly convex and uniformly smooth.

In Hilbert spaces, a different generalization of Wittmann's theorem was that of Bauschke~\cite{Bauschke(96)} to a finite family of nonexpansive maps $T_1, \cdots T_\ell$ satisfying:
\begin{equation}\label{bauschkecondition}\tag{+}
	\bigcap_{i=1}^{\ell}Fix(T_i)=Fix(T_{\ell}\cdots T_1)=\cdots =Fix(T_{\ell-1}\cdots T_1T_{\ell})
\end{equation}

Bauschke's schema is the particular case of \eqref{SnH} when one considers the sequence of nonexpansive maps $(S_n)$ defined cyclically by $S_n:=T_{[n+1]}$, for $n\in\N$,
where $[n]:=n \!\mod \ell$. Similarly to Wittmann's proof, Bauschke's generalization also follows from a sequential weak compactness argument. In \cite{FerreiraLeusteanPinto(19)}, a theoretical approach to eliminate this principle was developed, and subsequently a quantitative version of Bauschke's theorem was obtained. From and application of Theorem~\ref{SHtheo}, we then obtain a rate of metastability for a viscosity version of Bauschke's schema. In \cite{Koernlein(diss),Koernlein(descent)}, K\"ornlein also obtained rates of metastability for the viscosity version of Bauschke's schema. The rates obtained are different, since K\"ornlein's extraction follows from a direct proof (i.e. without a reduction to Bauschke's theorem). In the end, however, both rates look to be of a similar complexity: K\"ornlein still needed to carry out a convoluted quantitative treatment for the existence of (a proxy to) $z=P_F\phi z$ (see \cite[Problem 8.4.4.]{Koernlein(diss)}). K\"ornlein also considered Suzuki's comment on Bauschke's condition \cite{Suzuki(07)i}: without any commutativity assumption, the condition \eqref{bauschkecondition} already follows from having $\bigcap_{i=1}^{\ell}Fix(T_i)=Fix(T_{\ell}\cdots T_1)$. In \cite[Theorem 8.7.1]{Koernlein(diss)}, K\"ornlein proved a finitary version of this remark, and thus \cite[Theorem 6.10]{FerreiraLeusteanPinto(19)} can easily be adapted to the situation where the function $\tau$ instead satisfies
\[
\|x-T_{\ell}\cdots T_{1}(x)\|\leq \frac{1}{\tau(k)+1}\to \forall i<\ell\, \|x-S_i(x)\|\leq \frac{1}{k+1}.
\]

Our result entails the strong convergence of the viscosity version of Bauschke's schema with a Rakotch map towards a common fixed point. In \cite{Jung(06)}, Jung extended the convergence of the viscosity approximation method of finite families to the setting of (suitable) Banach spaces under a slightly more general condition than $(C3)$. Thus, we obtained a rate of metastability for Jung's schema in Hilbert spaces under the slightly stronger condition but for more general contractions $\phi$.

As a last application we look at proximal point algorithms, a setting in which sequences of nonexpansive maps are also naturally considered. Let $X$ be a Hilbert space, and $T:D(T)\subset X\to 2^X$ a monotone operator, i.e. $\langle x-x', y-y'\rangle \geq 0$ for any $y\in T(x), y'\in T(x')$. Furthermore, assume that $T$ is maximal in the sense that its graph is not properly contained in the graph of any other monotone operator. For a real number $\lambda >0$, the resolvent function $J_{\lambda}$ is the single-valued function $(\rm{Id}+\lambda T)^{-1}$. It is well known that $J_{\lambda}$ is nonexpansive and its fixed point set coincides with the set of zeros of $T$. First introduced by Martinet in \cite{Martinet(70)}, the so-called proximal point algorithm (\rm{PPA}) is defined by $x_{n+1}=J_{\lambda_n}(x_n)$, where $(\lambda_n)$ is a sequence of positive real numbers. It is known that the \rm{PPA} is only weakly convergent (\cite{Rockafellar(76)} and \cite{Guler(91)}), and so modifications to the algorithm were proposed. Introduced in \cite{KamimuraTakahashi(00)}, and independently in \cite{Xu(02)i}, we recall one such generalization: For $u, w_0\in X$, sequences of real numbers $(\alpha_n)\subset [0,1]$ and $(\lambda_n)\subset \R^+$, the iteration $(w_n)$ defined by
\begin{equation}\label{hppa}\tag{\rm{HPPA}}
	w_{n+1}=\alpha_nu+(1-\alpha_n)J_{\lambda_n}(w_n),
\end{equation}
is the Halpern type proximal point algorithm. Note that when $w_0=u$ this iteration is the particular case of \eqref{SnH} with $S_n=J_{\lambda_n}$ for all $n\in \N$.

Quantitative versions of the strong convergence of this iteration were carry out in \cite{Pinto(ta)}, \cite{LeusteanPinto(ta)} and \cite{DinisPinto(ta)}. Recently a generalization of \ref{hppa} to Banach spaces and accretive operators was studied in \cite{AoyamaToyoda(17)} and a rate of metastability was subsequently obtained in \cite{Kohlenbach(20)}. Pertaining to the Halpern type proximal point algorithm, this result is the most general so far, as strong convergence is established only asking that $(\alpha_n)\subset (0,1]$ be a sequence slowly converging to zero and that $\inf\lambda_n >0$. (The rate extracted in \cite{Kohlenbach(20)}, receives as an input a rate of metastability for a Browder type sequence, which can be instantiated with that of \cite{KohlenbachSipos(ta)}.) Hence, Theorem~\ref{SHtheo} entails a rate of metastability for a viscosity version of the Halpern type proximal point algorithm under appropriate conditions (see also \cite{AoyamaToyoda(19)}).

\subsection{Hybrid viscosity KM-iteration}

In this subsection, we give a rate of metastability for the \rm{vKM} iteration \eqref{vKM}, provided one has a rate of metastability for Browder sequences. Our approach makes use of a quantitative analysis of Theorem 5 in \cite{Xu(2020)} (for the case with no error terms). The following result computes a rate of metastability for the \rm{vKM} iteration from a rate of metastability for viscosity-Browder sequences.
\begin{theorem}\label{vKM-theorem}
	Let $C$ be a nonempty bounded convex subset of a hyperbolic space and let $b\in \N^*$ be a bound on the diameter of $C$. Consider $T$ a nonexpansive map on $C$, an $r$-contraction $\phi:C\to C$  with  $\delta\in (0,1)$ such that $r\leq 1-\delta$, and sequences $(\alpha_n), (\beta_n) \subset (0,1]$ satisfying $\sum^{\infty}_{n=1} \alpha_n\beta_n=\infty$ with a monotone rate of divergence $\mu_1$, and $\lim \frac{|\alpha_n-\alpha_{n-1}|}{\alpha_n^2\beta_n}=0$ with a monotone rate of convergence $\mu_2$. Let $(\tilde{x}_n)$ be the viscosity-Browder sequence defined with $T, \phi$ and the sequence $(\alpha_n)$ and $(x_n)$ be 
defined by \eqref{vKM}. 
Then $d(x_{n+1},\tilde{x}_n)\to 0$ with rate of convergence 
	\begin{equation*}
	\begin{gathered}
	\Xi(\varepsilon):=\sigma_1[\tilde{\mu}_1,b]\left(\varepsilon, \mu_2\left(\frac{\delta^2\varepsilon}{2b}\right)\right),\ \mbox{where} \\
	\sigma_1 \text{ is as in Lemma~\ref{Hlem1}},\ \mbox{and} \ 
	\widetilde{\mu_1}(k):=\mu_1\left(\left\lceil \frac{k}{\delta}\right\rceil\right).
	\end{gathered}
	\end{equation*}
In particular, if $(\tilde{x_n})$ is a Cauchy sequence with rate of 
metastability $\Psi$ then $(x_n)$ is  
a Cauchy sequence with metastability rate
	\begin{equation*}
\begin{gathered}	
\Omega(\varepsilon, f)=\Omega[b, \delta, \Psi,\mu_1,\mu_2](\varepsilon, f)=\max\{\Xi(\varepsilon/3), \Psi(\varepsilon/3, \hat{f})\}+1, \ \mbox{where} \\
\hat{f}(n):=f(\max\{ \Xi(\varepsilon/3),n\}+1).\\
\end{gathered}
	\end{equation*}

\end{theorem}

\begin{proof}
	Let $(\tilde{x}_n)$ denote for the viscosity-Browder sequence, i.e.
	\begin{equation*}
	\tilde{x}_n=(1-\alpha_n)T(\tilde{x}_n)\oplus \alpha_n\phi(\tilde{x}_n),\, n\in \N.
	\end{equation*}
	
	Using (W2) and (W4),
	\begin{align*}
	d(\tilde{x}_{n+1}, \tilde{x}_n)&\stackrel{\hphantom{\mathrm{W4}}}{\leq} d(\tilde{x}_{n+1}, W(T(\tilde{x}_n), \phi(\tilde{x}_n), \alpha_{n+1}))+d(W(T(\tilde{x}_n), \phi(\tilde{x}_n), \alpha_{n+1}), \tilde{x}_n)\\
	&\stackrel{\mathrm{W2}}{\leq} d(\tilde{x}_{n+1}, W(T(\tilde{x}_n), \phi(\tilde{x}_n), \alpha_{n+1})) + |\alpha_{n+1}-\alpha_n| d(T(\tilde{x}_n), \phi(\tilde{x}_n))\\
	&\stackrel{\mathrm{W4}}{\leq} (1-\alpha_{n+1}+\alpha_{n+1}r)d(\tilde{x}_{n+1}, \tilde{x}_n) + |\alpha_{n+1}-\alpha_n| d(T(\tilde{x}_n), \phi(\tilde{x}_n))\\
	&\stackrel{\hphantom{\mathrm{W4}}}{\leq} (1-\delta\alpha_{n+1})d(\tilde{x}_{n+1}, \tilde{x}_n)+|\alpha_{n+1}-\alpha_n|\cdot b,
	\end{align*}
	which implies
	\begin{equation*}
	d(\tilde{x}_{n+1},\tilde{x}_n)\leq \frac{b|\alpha_{n+1}-\alpha_n|}{\delta\alpha_{n+1}}.
	\end{equation*}
	
	Then, by \eqref{vKM} and using (W1) and (W4), we have for $n\ge 1$
	\begin{align*}
	d(x_{n+1}, \tilde{x}_n)&\stackrel{\mathrm{W1}}{\leq} (1-\beta_n)d(x_n, \tilde{x}_n)+\beta_nd((1-\alpha_n)T(x_n)\oplus\alpha_n\phi(x_n), \tilde{x}_n)\\
	&\stackrel{\mathrm{W4}}{\leq} (1-\beta_n)d(x_n, \tilde{x}_n)+\beta_n\left((1-\alpha_n)d(x_n, \tilde{x}_n)+\alpha_nrd(x_n, \tilde{x}_n)\right)\\
	&\stackrel{\hphantom{\mathrm{W4}}}{\leq} (1-\delta\alpha_n\beta_n)d(x_n, \tilde{x}_n)\leq (1-\lambda_n)d(x_n, \tilde{x}_{n-1})+\lambda_n\gamma_n,
	\end{align*}
	where $\lambda_n=\delta\alpha_n\beta_n$ and $\gamma_n=\frac{b|\alpha_n-\alpha_{n-1}|}{\delta^2\alpha_n^2\beta_n}$.
	
	Since $\mu_1$ is a rate of divergence for $(\sum\alpha_n\beta_n)$, we have that $\widetilde{\mu_1}$ is a rate of divergence for $(\sum\lambda_n)$. Hence by Lemma~\ref{Hlem1} (for an arbitrary $p$), we conclude that $d(x_{n+1}, \tilde{x}_n)$ converges to zero with rate of convergence $\Xi$.
	
	Let $\varepsilon >0$ and $f:\N\to\N$ be given. Since $\Psi$ is a rate of metastability for $(\tilde{x}_n)$, there exists $n'\leq \Psi(\varepsilon/3, \hat{f})$ be such that $d(\tilde{x}_i, \tilde{x}_j)\leq \frac{\varepsilon}{3}$, for $i,j\in[n', \hat{f}(n')]$. Take $n:=\max\{\Xi(\varepsilon/3), n')$. We have $n\leq \Omega(\varepsilon, f)-1$ and for $i, j\in [n, f(n+1)]$
	\begin{equation*}
	d(x_{i+1}, x_{j+1})\leq d(x_{i+1}, \tilde{x}_i)+d(\tilde{x}_i, \tilde{x}_j)+d(x_{j+1}, \tilde{x}_j)\leq \varepsilon,
	\end{equation*}
	since $[n, f(n+1)]\subset [n', \hat{f}(n')]$. This entails the result. 
\end{proof}

Due to Corollary~\ref{VBcor1}, we can just use a rate of metastability for the Browder sequence in the previous result (see also Lemma~\ref{lowerbound_meta}).
\begin{corollary} 
	Let $X,C,b,T,\phi,r,\delta,(\alpha_n),(\beta_n),\mu_1,\mu_2,(x_n), \sigma_1$ be as in Theorem~\ref{vKM-theorem}. Let $\theta_b$ be a monotone function satisfying \eqref{Bcondition} for $(y_n)$ the Browder sequence defined with $T$ and the sequence $(\alpha_n)$.
	Then $(x_n)$ defined by \eqref{vKM} is a Cauchy sequence with metastability rate
	\begin{equation*}
	\Omega(\varepsilon, f)=\Omega[b, \delta, \Psi,\mu_1,\mu_2](\varepsilon, f)
	\end{equation*}
	where
	\begin{equation*}
	\begin{gathered}
	\Psi=\Psi[b, \delta, \theta_b] \text{ is as in Corollary~\ref{VBcor1}},\\
	\Omega \text{ is as in Theorem~\ref{vKM-theorem}}.\\
	\end{gathered}
	\end{equation*}
\end{corollary}

As before we also get the following corollary 
\begin{corollary} 
Let $X,C,b,T,\phi,r,\delta,(\alpha_n),(\beta_n),\mu_1,\mu_2,(x_n), \sigma_1$ be as in Theorem \ref{vKM-theorem}. Let $\rho(\varepsilon)$ be a common (for all 
anchor points $u\in C$) rate of convergence for the Browder 
sequence $(y_n(u))$ (w.r.t. $(\alpha_n)$) of $T.$ 
Then $(x_n)$ defined by \eqref{vKM} is a Cauchy sequence with Cauchy rate
	\begin{equation*}
	\Omega(\varepsilon)=\Omega[b, \delta, \theta_b,\mu_1,\mu_2](\varepsilon)=\max\{\Xi(\varepsilon/3), \Psi(\varepsilon/3)\}+1,
	\end{equation*}
where
	\begin{equation*}
	\begin{gathered}
	\Xi(\varepsilon):=\sigma_1[\tilde{\mu}_1,b]\left(\varepsilon, \mu_2\left(\frac{\delta^2\varepsilon}{3b}\right)\right),\\
	\sigma_1 \text{ is as in Lemma~\ref{Hlem1}},\\
	\widetilde{\mu_1}(k):=\mu_1\left(\left\lceil \frac{k}{\delta}\right\rceil\right),\\
	\Psi=\Psi[b, \delta, \theta_b] \text{ is as in Corollary~\ref{VB-rate}}.\\
	\end{gathered}
	\end{equation*}
\end{corollary}

\subsection{Inexact algorithms}

In many applications of the algorithms studied above it is useful to allow for the iteration to contain error terms. In this subsection we provide rates of
 metastability of such relaxed versions under appropriate assumptions on the sequence of error terms. In the following, consider $(S_n)$ a family of nonexpansive maps, $\phi$ an $r$-contraction, $(\alpha_n)\subset (0,1]$ and $(\varepsilon_n)\subset \R_0^+$ a sequence of allowed errors. Assume that $\delta\in(0,1)$ is such that $r\leq 1-\delta$, and $b\in\N^*$ is a bound on the diameter of $C$.

\begin{proposition}\label{BrowderRelaxed}
	Let $(x_n)$ be a $(S_n)$-viscosity Browder sequence~\eqref{VSnB} and $(x'_n)$ be a sequence in $C$ satisfying for all $n\in\N$,
	\begin{equation*}
	d(x'_n,(1-\alpha_n)S_n(x'_n)\oplus\alpha_n\phi(x'_n))\leq \varepsilon_n.
	\end{equation*}
	If $\lim \frac{\varepsilon_n}{\alpha_n}=0$, then $\lim d(x_n,x'_n)=0$. Furthermore, if $\Psi$ is a rate of metastability for $(x_n)$ and $\rho$ is a rate of convergence for $\lim\frac{\varepsilon_n}{\alpha_n}=0$, then $\rho_{\delta}(\varepsilon):=\rho(\varepsilon\delta)$ is a rate of convergence for $\lim d(x_n,x'_n)=0$, and
	\begin{equation*}
	\Psi'(\varepsilon, f)=\max\{\rho_{\delta}(\varepsilon/3), \Psi(\varepsilon/3, f_{\rho_\delta, \varepsilon})\},
	\end{equation*}
	with $f_{\rho_\delta, \varepsilon}(n):=f(\max\{\rho_{\delta}(\varepsilon/3), n\})$, is a metastability rate for $(x'_n)$.
\end{proposition}

\begin{proof}
	Using (W4) we have
	\begin{align*}
	d(x_n,x'_n)&\leq (1-\alpha_n)d(S_n(x_n),S_n(x'_n))+\alpha_nd(\phi(x_n),\phi(x'_n))+\varepsilon_n\\
	&\leq (1-\alpha_n)d(x_n,x'_n)+\alpha_n rd(x_n,x'_n)+\varepsilon_n.
	\end{align*}
	So $d(x_n,x'_n)\leq \frac{\varepsilon_n}{\delta\alpha_n}$ and $\rho_{\delta}$ is a rate of convergence for $\lim d(x_n,x'_n)=0$.
	
	For a given $\varepsilon>0$ and $f:\N\to\N$, consider $n'\leq \Psi(\varepsilon/3, f_{\rho_\delta, \varepsilon})$ such that $d(x_i,x_j)\leq \frac{\varepsilon}{3}$, for all $i,j\in[n', f_{\rho_\delta, \varepsilon}(n')]$; and define $n:=\max \{\rho_{\delta}(\varepsilon/3), n'\}$. Then $n\leq \Psi'(\varepsilon, f)$ and
	\begin{equation*}
	d(x'_i,x'_j)\leq d(x'_i,x_i)+d(x_i,x_j)+d(x_j,x'_j)\leq \varepsilon,
	\end{equation*}
	for all $i,j\in [n, f(n)]$. This shows that $\Psi'$ is a rate of metastability for $(x'_n)$.	
\end{proof}

Similar one can study the metastability of a relaxed Halpern type iteration.
\begin{proposition}\label{HalpernRelaxed}
	Let $(x_n)$ be a $(S_n)$-viscosity Halpern iteration~\eqref{VSnH} and $(x'_n)$ be a sequence in $C$ satisfying $x'_0=x_0$ and for all $n\in\N$,
	\begin{equation*}
	d(x'_{n+1},(1-\alpha_n)S_n(x'_n)\oplus\alpha_n\phi(x'_n))\leq \varepsilon_n.
	\end{equation*}
	If $\sum \alpha_n=\infty$ and $\lim\frac{\varepsilon_n}{\alpha_n}=0$, then $\lim d(x_n,x'_n)=0$. Furthermore, if $\Psi$ is a rate of metastability for $(x_n)$, $A$ is a rate of divergence for $(\sum\alpha_n)$ and $\rho$ is a rate of convergence for $\lim\frac{\varepsilon_n}{\alpha_n}=0$, then $\Gamma(\varepsilon):=\sigma_1[\tilde{A}, b](\varepsilon, \rho(\frac{\delta\varepsilon}{2}))$ is a rate of convergence for $\lim d(x_n,x'_n)=0$, and
	\begin{equation*}
	\Psi'(\varepsilon, f)=\max\{\Gamma(\varepsilon/3), \Psi(\varepsilon/3, f_{\Gamma, \varepsilon})\},
	\end{equation*}
	is a metastability rate for $(x'_n)$, where $f_{\Gamma, \varepsilon}(n):=f(\max\{\Gamma(\varepsilon/3), n\})$, $\sigma_1$ is as in Lemma~\ref{Hlem1}, and $\tilde{A}(k):=A\left(\lceil\frac{k}{\delta}\rceil\right)$.
\end{proposition}

\begin{proof}
We have for all $n\in\N$,
\begin{align*}
d(x_{n+1},x'_{n+1})&\leq (1-\alpha_n)d(S_n(x_n),S_n(x'_n))+\alpha_n d(\phi(x_n),\phi(x'_n))+\varepsilon_n\\
&\leq (1-\alpha_n)d(x_n,x'_n)+\alpha_nrd(x_n,x'_n)+\varepsilon_n\\
&\leq (1-\delta\alpha_n)d(x_n,x'_n)+\delta\alpha_n\frac{\varepsilon_n}{\delta\alpha_n}.
\end{align*}
Since $\sum\delta\alpha_n=\infty$ and $\lim\frac{\varepsilon_n}{\delta\alpha_n}=0$, by Lemma~\ref{LemmaXu} we conclude that $d(x_n,x'_n)$ converges to zero. From the fact that $A$ is a rate of divergence for $(\sum\alpha_n)$ it follows that $\widetilde{A}$ is a rate of divergence for $(\sum\delta\alpha_n)$. From an application of Lemma~\ref{Hlem1} (for an arbitrary $p$) it follows that $\Gamma$ is a rate of convergence towards zero for $d(x_n,x'_n)$. One then argues that $\Psi'$ is a rate of metastability for $(x'_n)$ in the same way as in the proof of Proposition~\ref{BrowderRelaxed}.
\end{proof}

\begin{remark}\label{remark2}
One can also conclude $\lim d(x_n,x'_n)=0$ under the assumption $\sum \varepsilon_n<\infty$, using a different version of Xu's lemma -- e.g. {\rm\cite[Lemma 2.5]{Xu(02)i}}. In that case, a metastability rate can easily be obtained using a Cauchy modulus for $(\sum\varepsilon_n)$ -- see {\rm\cite[Lemma 3.4]{LeusteanPinto(ta)}} -- instead of the convergence rate $\rho$.
\end{remark}

We now consider additionally a sequence $(\beta_n)\subset (0,1]$ and discuss the metastability of the relaxed version of \rm{vKM}.

\begin{proposition}\label{vKMRelaxed}
	Let $(x_n)$ be a \rm{vKM} iteration given by \eqref{vKM} and $(x'_n)$ be a sequence in $C$ satisfying $x'_0=x_0$ and for all $n\in\N$,
	\begin{equation*}
	d(x'_{n+1},(1-\beta_n)x'_n\oplus \beta_n\left((1-\alpha_n)T(x'_n)\oplus \alpha_n \phi(x'_n)\right))\leq \varepsilon_n.
	\end{equation*}
	If $\sum \alpha_n\beta_n=\infty$ and $\lim\frac{\varepsilon_n}{\alpha_n\beta_n}=0$, then $\lim d(x_n,x'_n)=0$. Furthermore, if $\Psi$ is a rate of metastability for $(x_n)$, $A$ is a rate of divergence for $(\sum\alpha_n\beta_n)$ and $\rho$ is a rate of convergence for $\lim\frac{\varepsilon_n}{\alpha_n\beta_n}=0$, then $\Gamma(\varepsilon):=\sigma_1[\tilde{A}, b](\varepsilon, \rho(\frac{\delta\varepsilon}{2}))$ is a rate of convergence for $\lim d(x_n,x'_n)=0$, and
	\begin{equation*}
	\Psi'(\varepsilon, f)=\max\{\Gamma(\varepsilon/3), \Psi(\varepsilon/3, f_{\Gamma, \varepsilon})\},
	\end{equation*}
	is a metastability rate for $(x'_n)$, where $f_{\Gamma, \varepsilon}(n):=f(\max\{\Gamma(\varepsilon/3), n\})$, $\sigma_1$ is as in Lemma~\ref{Hlem1}, and $\tilde{A}(k):=A\left(\lceil\frac{k}{\delta}\rceil\right)$.
\end{proposition}

\begin{proof}
The proof of this result is similar to that of Proposition~\ref{HalpernRelaxed}.
\end{proof}

\begin{remark}\label{remark3}
Similar to {\rm Remark~\ref{remark2}}, it is also possible to conclude that $\lim d(x_n,x'_n)=0$ under the assumption $\sum \varepsilon_n<\infty$, and to obtain a rate of metastability rate for $(x'_n)$ using a Cauchy modulus for $(\sum\varepsilon_n)$.\\
Furthermore, both {\rm Propositions~\ref{HalpernRelaxed}} and {\rm~\ref{vKMRelaxed}} can be adapted to use a function $A'$ satisfying \eqref{condprod} instead of the rate of divergence $A$, in which case one makes use of the function $\sigma_2$ from {\rm Lemma~\ref{Hlem1}}, instead of the function $\sigma_1$.
\end{remark}

\section{Rates of convergence}

Let $(X,\|\cdot\|)$ be a normed linear space and $C\subseteq X$ some subset.
Suppose that $T:C\to C$ has at most one fixed point, i.e. 
\[ (1) \ \forall p_1,p_2 \in C\,(p_1=Tp_1\wedge p_2=Tp_2\to p_1=p_2). \]
We say that $T$ has uniformly at most one fixed point with modulus of 
uniqueness $\omega:(0,\infty)\to (0,\infty)$ if 
\[ (2) \ \forall \varepsilon>0\,\forall p_1,p_2\in C\, 
(\| p_1-Tp_1\|,\| p_2-Tp_2\|\le \omega (\varepsilon)\to \| p_1-p_2\|\le 
\varepsilon). \]
If $T$ is continuous and $C$ is compact, $(1)$ implies the existence 
of a modulus $\omega$ such that $(2)$ but in general $(2)$ is stronger 
than $(1).$ However, for large classes of uniqueness proofs one can 
actually extract from a proof of $(1)$ an explicit effective modulus 
$\omega$ satisfying $(2).$ We refer to \cite{Koh2008} for discussions on 
all this. Suppose now that we have for some iterative algorithm 
$(x_n)\subset C$ with $\lim \| x_n-Tx_n\|\to 0$ a rate 
$\tau:(0,\infty)\to \N$ of convergence, i.e. 
\[ \forall \varepsilon>0\, \forall n\ge \tau(\varepsilon)\, (\| 
x_n-Tx_n\|\le \varepsilon). \]
If $\omega$ is a modulus of uniqueness for $T$ in the sense 
above, then $\rho(\varepsilon):=\tau(\omega (\varepsilon))$ is a 
Cauchy rate for $(x_n).$ Hence if $X$ is complete, $C$ closed and 
$T$ is continuous, then $(x_n)$ converges to a (unique) fixed point 
$p$ of $T$ with rate of convergence $\rho.$
\\[1mm]
In the following we describe a class of nonexpansive operators 
$T$ for which a modulus $\omega$ can be computed if $X$ is 
uniformly convex. Consequently, we can then compute rates of 
convergence for $(x_n)$ defined by either Krasnoselskii-Mann 
or by Halpern iterations. 
\begin{definition}
\label{def-acc}
An operator $A:D(A)\to 2^X$ 
is accretive if for all $u\in Ax$ and $v\in Ay$ there exists some $j\in J(x-y)$ such that $\pair{u-v,j}\geq 0,$ where 
$J$ is the normalized duality mapping of $X.$
\end{definition}
Various strengthened forms of this notion have been considered in the 
literature which guarantee that $A$ has at most one zero (see e.g. \cite{G-Falset(05)}):

\begin{definition}\label{def-unacc}
\begin{enumerate}

\item Let $\psi:[0,\infty)\to [0,\infty)$ be a continuous function with $\psi(0)=0$ and $\psi(x)>0$ for $x>0$. Then an operator $A:D(A)\to 2^X$ is 
$\psi$-strongly accretive if
\begin{equation*}
\forall (x,u),(y,v)\in A\; \exists j\in J(x-y)\; (\pair{u-v,j}\geq \psi(\norm{x-y})\norm{x-y}).
\end{equation*}

\item Let $\phi:[0,\infty)\to [0,\infty)$ be a continuous function with $\phi(0)=0$ and $\phi(x)>0$ for $x>0$. Then an operator $A:D(A)\to 2^X$ is 
uniformly $\phi$-accretive if
\begin{equation*}
\forall (x,u),(y,v)\in A\; \exists j\in J(x-y)\; (\pair{u-v,j}\geq \phi(\norm{x-y})).
\end{equation*}

\end{enumerate}
\end{definition}
In the case of $\psi$-strongly accretive operators, $\psi$ is often assumed 
to be strictly increasing in addition (see e.g. \cite{Xu(2020)}). 
\\[1mm]
In \cite[Section 2.1]{KohKou(2015.0)} it has been exhibited that all what 
is needed to get a modulus of uniqueness for the property of being a 
zero of $A$ is the following (which is implied by the aforementioned 
concepts of $\psi,\phi$-accretivity):
\begin{definition}
  An accretive operator $A\hspace*{-1mm}:\hspace*{-1mm}D(A)\to 2^X$ is uniformly accretive  if
\begin{equation*}
(\ast) \ \left\{ \ \begin{aligned}
&\forall \varepsilon,K>0\; \exists \delta>0\; \forall (x,u),(y,v)\in A\\
&(\norm{x-y}\in [\varepsilon,K]\to \exists j\in J(x-y)\;(\pair{u-v,j}\geq \delta)).
\end{aligned} \right.
\end{equation*}
Any function $\Theta_{(\cdot)}(\cdot):(0,\infty)\times (0,\infty)\to (0,\infty)$ such that $\delta:=\Theta_K(\varepsilon)$ satisfies $(\ast)$ for all $\varepsilon,K>0$ is called a \emph{modulus of uniform accretivity} for $A$.
\end{definition}
\begin{remark} If $A$ is assumed to have a zero $q$, then the definition above 
can be modified to `uniform accretivity at zero' 
\begin{equation*}
(\ast\ast) \ \left\{ \ \begin{aligned}
&\forall \varepsilon,K>0\; \exists \delta>0\; \forall (x,u)\in A\\
&(\norm{x-q}\in [\varepsilon,K]\to \exists j\in J(x-q)\;(\pair{u,j}\geq \delta))
\end{aligned} \right.
\end{equation*}
which still is sufficient to get a modulus of uniqueness for the property 
of being a zero of $A$ (see {\rm \cite[Section 2.1]{KohKou(2015.0)}}).
\end{remark}
In \cite{BrezisSibony}, the following concept of uniform accretivity is 
introduced and shown to imply the uniqueness of zeroes if $X$ is a 
uniformly convex normed space (see also \cite{Gwinner(78)} and 
\cite{Xu(2020)} for results involving this and related notions):
\begin{definition}\label{definition-Brezis-Sibony}
$A$ is uniformly accretive in the sense of {\rm \cite{BrezisSibony}} 
if there is 
a strictly increasing function $\varphi:[0,\infty)\to\R$ with 
$\lim_{t\to\infty}\varphi(t)=\infty$ such that
\begin{equation*}
\forall (x,u),(y,v)\in A\; \exists j\in J(x-y)\; (\pair{u-v,j}\geq \left(\varphi
(\norm{x})-\varphi(\norm{y})\right)\cdot(\norm{x}-\norm{y})).
\end{equation*}
\end{definition}
Consider in the following a nonexpansive selfmapping $T:C\to C$ of a 
bounded convex subset $C\subset X$ of a normed space
$X$ such that $A:=I-T$ is an (single-valued) accretive operator with 
$D(A)=C.$ If $A$ admits a modulus of uniqueness for the property 
of being a zero, then the same modulus is also a modulus of uniqueness 
for being a fixed point of $T.$ We now extract from the uniqueness proof 
given in \cite{BrezisSibony} a modulus of uniqueness (again in the case of 
uniformly convex normed spaces) which only depends on a modulus function
$\Omega:(0,\infty)^2\to (0,\infty)$ such that 
\[ (+) \,\left\{ \ba{l}
\forall (x,u),(y,v)\in A\;\forall b,\varepsilon>0\; 
\exists j\in J(x-y)\\[1mm] \hspace*{2mm} 
 (\max\{ \| x\|,\| y\|\le b\wedge \left| \| x\| -\| y\|\right| \ge \varepsilon 
\to \pair{u-v,j} >\Omega(\varepsilon,b)). \ea \right. \]
\begin{proposition}
If $A$ is uniformly accretive in the sense of {\rm \cite{BrezisSibony}}, then 
there exists a modulus $\Omega$ satisfying $(+).$
\end{proposition} 
\begin{proof} Let $\varphi$ be as in Definition \ref{definition-Brezis-Sibony}.
Let $b,\varepsilon >0.$ It suffices to show that 
\[ \exists \delta>0 \,\forall x\in [0,b]\,(\varphi(x+\varepsilon)-\varphi(x)>\delta) \]
since $\varphi(y)\ge \varphi(x+\varepsilon)$ for $y\ge x+\varepsilon$ and we can then take $\Omega(\eps,b):=\eps\cdot\delta$.
Suppose otherwise, i.e. 
\[ \forall n\in\N \,\exists x_n\in [0,b]\, (\varphi(x_n+\varepsilon)-\varphi(x_n)\le 
\frac{1}{n+1}). \]
Let $x$ be a limit point of $(x_n).$ Then for arbitrary large $n$ we have $x_n\le 
x+\frac{\varepsilon}{3}$ and $x+\frac{2\varepsilon}{3}\le x_n+\varepsilon$ and 
so 
\[ \varphi\left( x+\frac{2\varepsilon}{3}\right) -\varphi\left( x+\frac{\varepsilon}{3}\right) 
\le \varphi (x_n+\varepsilon)-\varphi (x_n)\le\frac{1}{n+1}. \]
Thus $\varphi\left(x+\frac{2\varepsilon}{3}\right) -\varphi\left( x+\frac{\varepsilon}{3}\right)\le 0$ 
which contradicts $\varphi$ being strictly increasing. 
\end{proof} 
In this section, we will compute a modulus of uniqueness for approximate fixed points 
of $T=I-A$ in terms of a given modulus of uniform convexity for $X$ and a modulus 
$\Omega$ for $A.$ This then gives us explicit rates of convergence for Browder-type 
sequences, Krasnoselskii-Mann iterations and Halpern iterations of $T$ and so - by the 
results in the previous section - also such rates for the viscosity versions of 
these algorithms. \\
In the following $C\subset X$ is a bounded convex subset of a uniformly 
convex normed space $(X,\|\cdot\|).$ Let $b\ge \| x\|$ for all $x\in C$ and 
$\eta:(0,2]\to (0,1]$ be a modulus 
of uniform convexity for $X,$ i.e. 
\[ \forall \varepsilon\in (0,2] \,\forall x,y\in X \ \left(
\| x\|,\|y\|\le 1\,\wedge\, \| x-y\|\ge \varepsilon\to 
\left\| \frac{1}{2} (x+y)\right\| \le 1-\eta(\varepsilon)\right).\]
Note that this implication holds trivially if $\varepsilon >2$ and so 
we can extend $\eta$ to $(0,\infty)$ e.g. by putting it $:=1$ for 
$\varepsilon>2.$ \\ 
$T:C\to C$ is a nonexpansive mapping such that $A:=I-T$ is uniformly 
accretive in the sense of $(+)$ above with a respective modulus $\Omega$. In the following, we write $\Omega(\eps):=\Omega(\eps, b)$.

\begin{lemma}\label{lemma1}
Let $x_1,x_2\in C$ be $\frac{\Omega(\varepsilon)}{4b}$-approximate fixed 
points of $T.$ Then $\left| \| x_1\| -\| x_2\|\right| \le \varepsilon.$
\end{lemma}
\begin{proof} 
By assumption we have $\| A(x_1)\|,\| A(x_2)\|\le 
\frac{\Omega(\varepsilon)}{4b}.$ W.l.o.g $\| x_1\|\ge \| x_2\|.$ 
Let $j\in J(x_1-x_2)$ be as in $(+).$ Assume that $\| x_1\| -\| x_2\| >\eps.$ 
Then 
\[ \ba{l} \Omega(\varepsilon)=\frac{\Omega(\varepsilon)}{2b}\cdot 2b\ge 
\left( \| Ax_1\|+\| Ax_2\|\right)\cdot \| x_1-x_2\|
\\ \hspace*{2.4cm}
\ge \| Ax_1-Ax_2\|\cdot\| x_1-x_2\|\ge \langle Ax_1-Ax_2,j\rangle 
>\Omega(\varepsilon). \ea \]
Hence $\left| \| x_1\|-\| x_2\|\right| \le\eps.$ 
\end{proof}
\begin{lemma}\label{lemma2}
If $x_1,x_2\in C$ are $\frac{\eps\cdot\eta(\eps/2b)}{4}$-approximate fixed 
points of $T,$ then $\frac{x_1+x_2}{2}$ is an $\eps$-approximate fixed 
point of $T.$ \\ 
If $\eta(\varepsilon)$ can be written as $\varepsilon\cdot \tilde{\eta}(\varepsilon)$ with $0<\varepsilon_1\le\varepsilon_2\leq 2\to \tilde{\eta}(\varepsilon_1)
\le \tilde{\eta}(\varepsilon_2),$ then we can replace $\frac{\eps\cdot\eta(\eps/2b)}{4}$ by 
$\frac{\eps}{2}\cdot\tilde{\eta}(\eps/2b)=b\cdot\eta(\eps/2b).$ 

\end{lemma}
\begin{proof} The proof follows the pattern of the proofs of Lemmas 2.2 and 
2.3 in \cite{Kohlenbach(11)}.  Claim: 
\[ \forall \eps>0\,\forall a,x,y\in C\,\left( \left\| a-\frac{x+y}{2}\right\| 
>K-\frac{\eps\cdot\eta(\eps/2b)}{2} \to \| x-y\| <\eps\right), \] 
where $K:=\max\{ \| a-x\|,\| a-y\|\}\le 2b.$ \\
Proof of claim: w.l.o.g we may assume that $K\ge \eps/2,$ since, otherwise, 
$\| x-y\| < \frac{\eps}{2}+\frac{\eps}{2}=\eps.$ Consider 
\[ \tilde{x}:=\frac{a-x}{K}, \tilde{y}:=\frac{a-y}{K}.\] Then 
$\tilde{x},\tilde{y}\in \overline{B}_1(0).$ Assume that 
\[  \left\| a-\frac{x+y}{2}\right\| 
>K-\frac{\eps\cdot\eta(\eps/2b)}{2} \ge K-K\cdot\eta(\eps/2b) \] and so 
\[ \left\|\frac{\tilde{x}+\tilde{y}}{2}\right\| =\frac{1}{K}
 \left\| a-\frac{x+y}{2}\right\|  >1-\eta(\eps/2b). \]
Then by the definition of $\eta$ 
\[ \frac{1}{K} \| x-y\| =\| \tilde{x}-\tilde{y}\| <\frac{\eps}{2b}\le
\frac{\eps}{K}. \]
Proof of the lemma: by assumption 
\[ (1) \ \bigwedge^2_{i=1} \| x_i-Tx_i\| \le \frac{\eps\cdot\eta(\eps/2b)}{4}.\]
Define $x:=\frac{x_1+x_2}{2}.$ 
For $i=2,$ we get 
\[ (2) \ \left\{\ba{l} \left\| x_2-\frac{x+Tx}{2}\right\| \le \frac{1}{2}
\| x_2-x\| +\frac{1}{2}\| x_2-Tx\|\\ \le
 \frac{1}{2}
\| x_2-x\| +\frac{1}{2}\| Tx_2-Tx\|+\frac{\eps\cdot\eta(\eps/2b)}{8} \\
\le
 \frac{1}{2}
\| x_2-x\| +\frac{1}{2}\| x_2-x\|+\frac{\eps\cdot\eta(\eps/2b)}{8}
= \| x_2-x\| +\frac{\eps\cdot\eta(\eps/2b)}{8}. \ea \right. \]
This implies that 
\[ (3) \ \left\| x_1-\frac{x+Tx}{2}\right\| \ge \| x_1-x\| 
-\frac{\eps\cdot\eta(\eps/2b)}{8}, \]
since, otherwise, we get the contradiction  
\[ \ba{l} 
\| x_1-x_2\| \le  \left\| x_1-\frac{x+Tx}{2}\right\|
+   \left\| x_2-\frac{x+Tx}{2}\right\| \\ <
\| x_1-x\| -\frac{\eps\cdot\eta(\eps/2b)}{8} +\| x_2-x\| +\frac{\eps\cdot\eta(\eps/2b)}{8} \\ =\frac{1}{2}\| x_1-x_2\| +\frac{1}{2}\| x_1-x_2\| =
\| x_1-x_2\|. \ea \]
By $(1)$ applied to $i=1$ we also have 
\[ (4) \ \| x_1-Tx\|\le \| Tx_1-Tx\| +\frac{\eps\cdot\eta(\eps/2b)}{4}
\le \| x_1-x\| +\frac{\eps\cdot\eta(\eps/2b)}{4}. \] 
With $K:=max\{ \| x_1-x\|,\| x_1-Tx\|\le 2b,$ $(3)$ and $(4)$ yield that 
\[ \left\| x_1-\frac{x+Tx}{2}\right\| \ge \| x_1-x\|-\frac{\eps\cdot\eta(\eps/2b)}{8} \ge K-\frac{\eps\cdot\eta(\eps/2b)}{4}  
-\frac{\eps\cdot\eta(\eps/2b)}{8}>K-\frac{\eps\cdot\eta(\eps/2b)}{2}. \] 
The 1st part of the 
lemma now follows by the claim with $x_1, x, Tx$ as $a,x,y.$  \\ 
For the 2nd part, one first has to observe that the proof of the claim also establishes the claim for 
\[ K\cdot\eta(\eps/K) =\eps\cdot \tilde{\eta}(\eps/K)\ge 
 \eps\cdot \tilde{\eta}(\eps/2b)\] instead of 
$\frac{\eps}{2}\cdot \eta(\eps/2b).$ 
Using this,  
\[ (1)' \ \bigwedge^2_{i=} \| x_i-Tx_i\| \le 
\frac{\eps}{2} \cdot \tilde{\eta}(\eps/2b) \]
now implies that $\| x-Tx\|<\eps$ as in the proof above.
\end{proof}

The next lemma is a slight adaption of Lemma 3.2 in \cite{KohLeu09} 
(we include the brief proof for completeness):
\begin{lemma}\label{lemma3}
If $x_1,x_2\in C$ with $\| x_1\|\ge \| x_2\|$ and $\| x_1-x_2\| >\eps,$ 
then $\left\| \frac{x_1+x_2}{2}\right\| <\| x_1\| -\beta(b,\eps),$ where 
\[ \beta(b,\eps):=\frac{\eps}{2}\cdot \eta(\eps/b). \]
If $\eta(\varepsilon)$ can be written as $\varepsilon\cdot \tilde{\eta}(\varepsilon)$ with $0<\varepsilon_1\le\varepsilon_2\leq 2\to \tilde{\eta}(\varepsilon_1)
\le \tilde{\eta}(\varepsilon_2),$ then we can replace $\beta$ by 
$\beta'(b,\eps):=\eps\cdot\tilde{\eta}(\eps/b)=b\cdot\eta(\eps/b).$ 
\end{lemma}
\begin{proof} 
 $\| x_1\|\ge \| x_2\|$ and $\| x_1-x_2\| >\eps$ implies that 
\[ (1) \ \| x_1\| >\frac{\eps}{2}>0.\]
Define $\tilde{x}_1:=\frac{x_1}{\| x_1\|}, \tilde{x}_2:=\frac{x_2}{\| x_1\|}.$
Then $\|\tilde{x}_2\|\le \| \tilde{x}_1\| =1$ and 
\[ (2) \ \| \tilde{x}_1 -\tilde{x}_2\| =\frac{1}{\| x_1\|} \| x_1-x_2\| >
\frac{\eps}{\| x_1\|}\ge \frac{\eps}{b}.\]
Since $\eta$ is a modulus of uniform convexity we get
\[ (3) \  \left\| \frac{x_1+x_2}{2}\right\| =\| x_1\|\cdot\left\| 
\frac{\tilde{x}_1+\tilde{x}_2}{2}\right\| \le \| x_1\| -\| x_1\|\cdot 
\eta(\eps/b) 
\stackrel{(1)}{<} \| x_1\| -\frac{\eps}{2}\cdot \eta(\eps/b).  \]
The 2nd claim follows from 
\[ (3)' \ \left\{ \ba{l}  \left\| \frac{x_1+x_2}{2}\right\| =\| x_1\|\cdot\left\| 
\frac{\tilde{x}_1+\tilde{x}_2}{2}\right\| < \| x_1\| -\| x_1\|\cdot \eta(\eps/\| x_1\|) \\ 
= \| x_1\| -\eps\cdot \tilde{\eta}(\eps/
\| x_1\| )\le \| x_1\| -\eps\cdot \tilde{\eta}(\eps/b). \qedhere \ea 
\right. \]
\end{proof}
\begin{proposition}\label{modulus-of-uniqueness}
Let 
\[ \omega_b(\eps):=\frac{1}{16b}\Omega\left( \frac{\eps}{2}\cdot\eta\left( \frac{\eps}{b}
\right)\right)\cdot\eta\left(\frac{1}{8b^2}\Omega\left( \frac{\eps}{2}\cdot 
\eta\left( \frac{\eps}{b}
\right)\right)\right). \]
If $x_1,x_2\in C$ are $\omega_b(\eps)$-approximate fixed points of $T,$ then $\| x_1-x_2\|\le\eps,$
i.e. $\omega_b$ is a modulus of uniqueness for being a fixed point of $T.$
\\ If $\eta(\eps)=\eps\cdot \tilde{\eta}(\eps)$ with $\tilde{\eta}$ as 
in Lemma \ref{lemma3}, then we can improve $\omega_b$ to 
\[ \omega_b(\eps):=\min\left\{ 
b\cdot\eta\left( \frac{\Omega\left( b\cdot 
\eta\left( \frac{\eps}{b}\right)\right)}{8b^2}\right),
\frac{\Omega\left( b\cdot 
\eta\left( \frac{\eps}{b}\right)\right)}{4b}\right\}. \]
\end{proposition}
\begin{proof} 
W.l.o.g. $\| x_1\|\ge \| x_2\|.$ Assume that $\| x_1-x_2\|>\eps.$ By 
Lemma \ref{lemma2} $\frac{x_1+x_2}{2}$ (and by assumption also $x_1$ 
since $\eta(\ldots)\le 1$) is a $\frac{1}{4b}\Omega\left(
\frac{\eps}{2}\cdot\eta\left(\frac{\eps}{b}\right)\right)$-approximate 
fixed point of $T.$ Hence by Lemma \ref{lemma1} (applied to 
$\frac{x_1+x_2}{2}$ as $x_2$) 
\[ \| x_1\| -\left\| \frac{x_1+x_2}{2}\right\| \le \frac{\eps}{2}\cdot\eta
\left(\frac{\eps}{b}\right). \] 
By Lemma \ref{lemma3}, on the other hand, we have that 
\[ \left\| \frac{x_1+x_2}{2}\right\| <\| x_1\| -\frac{\eps}{2}\cdot \eta\left( \frac{\eps}{b}\right)  \] 
which gives a contradiction. 
The proof of the 2nd claim is analogous making use of the 2nd claims 
in lemmas \ref{lemma2} and \ref{lemma3}.
\end{proof}
A first application of Proposition \ref{modulus-of-uniqueness} gives a 
Cauchy rate for the path $(x_a)_{\alpha\in (0,1)}$ of points in $C$ such 
that (for given $c\in C$)  $x_{\alpha}=(1-\alpha)T(x_{\alpha})+\alpha c.$
\begin{theorem}
Let $C,X,\eta,b,T,\Omega$ be as before. Then 
\[ \forall \eps>0\,\forall \alpha_1,\alpha_1 \in (0,1)\ \left(\bigwedge^2_{i=1} 
\alpha_i\le \frac{\omega_b(\eps)}{2b}\to \| x_{\alpha_1}-x_{\alpha_2}\| \le
\eps\right). \] 
\end{theorem}
\begin{proof} Let for $i=1,2$ be $\alpha_i\in (0,1)$ be such that $\alpha_i 
\le\frac{\omega_b(\eps)}{2b}.$ Then
\begin{align*}
	\| T(x_{\alpha_i})-x_{\alpha_i}\| &=\| T(x_{\alpha_i}) -T(x_{\alpha_i})
	+\alpha_iT(x_{\alpha_i})-\alpha_i c\|\\
	&=\alpha_i\| T(x_{\alpha_i})-c\|\le
	\alpha_i \cdot 2b\le\omega_b(\eps).
\end{align*}
Hence by Proposition \ref{modulus-of-uniqueness} we get that 
$\| x_{\alpha_1}-x_{\alpha_2}\|\le \eps.$
\end{proof}

As another application of Proposition \ref{modulus-of-uniqueness} we get 
the following quantitative form of Theorem 1 in \cite{Xu(2020)}:
\begin{theorem}
Let $C,X,\eta,b,T,\Omega$ be as before. Let $(\beta_n)\subset [0,1]$ be 
such that $\sum^{\infty}_{n=0} \beta_n(1-\beta_n)=\infty$ with a rate of divergence 
$\gamma.$ Define for $x_0\in C$
\[ x_{n+1}:=(1-\beta_n)x_n+\beta_nTx_n. \]
Then $(x_n)$ is a Cauchy sequence with rate 
\[\forall \eps>0\,\forall n,m\ge\gamma\left(\left\lceil 
\frac{4b^2}{\pi\cdot (\omega_b(\eps))^2}\right\rceil\right)\ 
\left( \| x_n-x_m\|\le \eps\right), \]  where $\omega_b$ is the modulus of 
uniqueness from Proposition \ref{modulus-of-uniqueness}. 
For $X$ complete and $C$ closed, $(x_n)$ converges with this 
rate to the unique fixed point of $T.$
\end{theorem}
\begin{proof}
By \cite{Cominettietal} we have for all $n$
\[ (1) \ \| x_n-Tx_n\| \le \frac{2b}{\sqrt{\pi\cdot\sum^n_{i=0} 
\beta_i(1-\beta_i)}}. \]
Now let $n\ge\gamma\left(\left\lceil 
\frac{4b^2}{\pi\cdot (\omega_b(\eps))^2}\right\rceil\right),$ then 
\[ (2) \ \sum^{n}_{i=0} 
\beta_i(1-\beta_i)\ge \frac{4b^2}{\pi\cdot\omega_b(\eps)^2} \] and so by $(1)$
\[ \| x_n-Tx_n\|\le \omega_b(\eps). \]
By Proposition \ref{modulus-of-uniqueness}, the claim on the Cauchy 
modulus follows. If $X$ is complete and $C$ is closed, then -  by (1) - 
the limit $p$ of $(x_n)$ is a fixed point of $T$ (which by 
Proposition \ref{modulus-of-uniqueness} is unique).
\end{proof}

\begin{theorem} Let $C,X,\eta,b,T,\Omega$ be as before.
Assume that $\alpha_n\in (0,1)$ with 
\[  \lim_{n\to\infty} \alpha_n=0, \ \sum^{\infty}_{n=1} |\alpha_{n+1}-\alpha _n| \ \mbox{converges}, \ 
\prod^{\infty}_{n=1} (1-\alpha_n)=0.\] 
Let $\alpha$ be a rate of convergence of $(\alpha_n)$, $\beta$ be a 
Cauchy modulus of $s_n:=\sum_{i=1}^n|\alpha_{i+1}-\alpha_i|$ and $\theta$ 
be a rate of convergence of $\prod_{n=1}^\infty (1-\alpha_{n+1})=0$ 
towards $0.$ Define the Halpern iteration of $T$ with starting point $x\in C$ 
and anchor $u\in C$ by 
\[ x_0:=x\, \ \ x_{n+1}:=(1-\alpha_{n})T(x_n)+\alpha_{n}u. \]
Then $(x_n)$ is a Cauchy sequence with rate
\[ \forall \eps>0\,\forall n,m\ge\Phi(\omega_b(\eps)) \
\left( \| x_n-x_m\|\le \eps\right), \ \mbox{where}\]
\[ \Phi(\eps,b,\theta,\alpha,\beta,D) =  
\max\left\{\theta\left(\frac{D\eps}{4b}\right)+1,\alpha\left(
\frac{\eps}{4b}\right)\right\}+1, 
\]
and \[ 0<D\le  \prod_{n=1}^{\beta(\eps/8b)}(1-\alpha_{n+1}).\]
For the choice $\alpha_{n}:=\frac{1}{n+1},$ we may take 
\[ \Phi(\eps,b):=\left\lceil \frac{4b}{\eps}+\frac{32b^2}{\eps^2}\right\rceil 
.\]
\end{theorem}
\begin{proof} The theorem follows from  
Proposition \ref{modulus-of-uniqueness} combined with 
\cite[Proposition 6.2,Corollary 6.3]{KohlenbachLeustean(12)}.
\end{proof}

%

{\bf Acknowledgments:} Both authors have been 
supported by the German Science Foundation (DFG 
Project KO 1737/6-1).

The second author was also supported by FCT - Funda{\c c}{\~ a}o para a Ci{\^ e}ncia e a Tecnologia, under
the projects UIDB/04561/2020 and UIDP/04561/2020, and the research center CMAFcIO - Centro
de Matem{\' a}tica, Aplica{\c c}{\~ o}es Fundamentais e Investiga{\c c}{\~ a}o Operacional.

\end{document}